\documentclass[12pt, a4paper, reqno]{amsart}
\usepackage{amscd, amssymb}
\usepackage{amsthm}
\usepackage{tikz}
\usetikzlibrary{arrows,positioning} 
\usepackage{mathrsfs}
\usepackage{bbm}
\usepackage{ stmaryrd }

\usepackage{multirow}
\usepackage{makecell}

\usepackage{hyperref}
\hypersetup{
	colorlinks=true,
	linkcolor=blue,
	filecolor=magenta,  
	urlcolor=cyan,
}

\usepackage{hyperref}
\usepackage[capitalise]{cleveref}

\def\Aut{\operatorname{Aut}}

\def\ker{\operatorname{ker}}

\def\dim{\operatorname{dim}}
\def\ad{\operatorname{ad}}
\def\rk{\operatorname{rk}}

\def\Ber{\operatorname{Ber}}

\def\Ind{\operatorname{Ind}}

\def\Ad{\operatorname{Ad}}

\def\Hom{\operatorname{Hom}}

\def\Rep{\operatorname{Rep}}

\def\C{\mathbb{C}}

\def\R{\mathbb{R}}
\def\N{\mathbb{N}}
\def\Z{\mathbb{Z}}

\def\F{\mathbb{F}}

\def\G{\mathbb{G}}

\def\TT{\mathcal{T}}
\def\LL{\mathcal{L}}
\def\OO{\mathcal{O}}
\def\KK{\mathcal{K}}
\def\BB{\mathcal{B}}
\def\HH{\mathcal{H}}
\def\UU{\mathcal{U}}

\def\XX{\mathcal{X}}
\def\EE{\mathcal{E}}
\def\FF{\mathcal{F}}
\def\VV{\mathcal{V}}
\def\TT{\mathcal{T}}
\def\DD{\mathcal{D}}
\def\GG{\mathcal{G}}
\def\NN{\mathcal{N}}
\def\PP{\mathcal{P}}

\def\CC{\mathcal{C}}

\def\a{\mathfrak{a}}
\def\b{\mathfrak{b}}
\def\c{\mathfrak{c}}
\def\m{\mathfrak{m}}
\def\n{\mathfrak{n}}
\def\p{\mathfrak{p}}
\def\q{\mathfrak{q}}
\def\g{\mathfrak{g}}
\def\h{\mathfrak{h}}
\def\r{\mathfrak{r}}
\def\k{\mathfrak{k}}
\def\l{\mathfrak{l}}
\def\s{\mathfrak{s}}
\def\o{\mathfrak{o}}

\def\z{\mathfrak{z}}

\def\vol{\mbox{vol}}

\def\d{\partial}

\def\ol{\overline}

\def\Spec{\text{Spec}}

\def\Ext{\text{Ext}}

\def\sub{\subseteq}

\def\xto{\xrightarrow}

\newtheorem{thm}{Theorem}[section]
\newtheorem{cor}[thm]{Corollary}
\newtheorem{lemma}[thm]{Lemma}
\newtheorem{prop}[thm]{Proposition}

\theoremstyle{definition}
\newtheorem{definition}[thm]{Definition}
\theoremstyle{remark}
\newtheorem{remark}[thm]{Remark}
\newtheorem{example}[thm]{Example}

\numberwithin{equation}{section}

\usepackage{relsize}
\usepackage{fancyhdr}
\usepackage[all,cmtip]{xy}

\begin{document}

	\title{Sylow theorems for supergroups}
	
	\author{Vera Serganova, Alexander Sherman, Dmitry Vaintrob}
	
	\begin{abstract}
		We introduce Sylow subgroups and $0$-groups to the theory of complex algebraic supergroups, which mimic Sylow subgroups and $p$-groups in the theory of finite groups.  We prove that Sylow subgroups are always $0$-groups, and show that they are unique up to conjugacy.  Further, we give an explicit classification of $0$-groups which will be very useful for future applications. Finally, we prove an analogue of Sylow's third theorem on the number of Sylow subgroups of a supergroup.
	\end{abstract}
	\maketitle
	\pagestyle{plain}
	
	\section{Introduction}
	
	In this paper, we introduce \emph{0-groups} and \emph{Sylow subgroups} of supergroups, which are super-analogues of $p$-groups and Sylow subgroups of finite groups. We classify $0$-groups, and establish analogues of several basic statements concerning Sylow subgroups in the supergroup setting.  These include the uniqueness of Sylow subgroups up to conjugacy.
	
	\subsection{Finite groups} 
	Sylow subgroups, and more generally $p$-local subgroups, play a central, if mysterious, role in the representation theory of finite groups. As a simple example, consider the McKay conjecture (see \cite{N}): the number of simple representations of a finite group $G$ whose dimension is coprime to $p$ should be the same as that of the normalizer of a $p$-Sylow subgroup.

	Sylow subgroups, and more generally $p$-local subgroups, are particularly important when studying modular (i.e.~characteristic $p$) representations of a finite group $G$. Any block of representations determines a defect subgroup $H$ (which is a $p$-group), and restriction to the the defect is faithful on derived categories.  For instance, a $p$-Sylow subgroup is a defect subgroup of the principal block, and thus detects all the homological complexity of $\Rep_{\Bbbk}G$ ($\Bbbk$ an algebraically closed field of characteristic $p$).

	Such guiding principles in modular representation theory often go under the heading of local representation theory. Famous results and conjectures arising from these ideas include Green’s correspondence, the Broué abelian defect conjecture, and the Alperin weight conjecture. In addition, the study of $p$-local subgroups plays a key role in understanding the Balmer spectrum of $\Rep_{\Bbbk}G$.
	
	\subsection{Supergroups} One may initially look suspiciously upon attempted analogies between finite groups and supergroups.  Supergroups are geometric in nature, usually involving high-dimensional group manifolds, and are not determined by their closed points.  Nevertheless, as we seek to demonstrate here and in future work, there are clear (if not always perfect) analogies, and much can be understood by naively following them.
	
	Perhaps the first indications of an analogy appear in the works of Boe, Kujawa and Nakano on cohomological support varieties (\cite{BKN}).  For a simple Lie superalgebra $\g$, they define a detecting subalgebra $\mathfrak{f}\sub\g$, which conjecturally controls the cohomological support variety of a $\g$-module, similarly to how a Sylow subgroup would.  Their ideas led them more recently to a description of the Balmer spectrum for $\Rep\g\l(m|n)$, see \cite{BKN2}.  In a future work, we will use our ideas to answer many questions raised in \cite{BKN} about cohomological support varieties
	
	\subsection{Sylow subgroups}
	We take a broader approach to these questions.  Let $\GG$ be a complex algebraic supergroup such that $\GG_0$, the even underlying algebraic group of $\GG$, has reductive identity component. Such supergroups are called quasireductive. Then $\Rep\GG$, the category of all $\GG$-modules (not necessarily finite-dimensional) is a Frobenius category, just as $\Rep_{\Bbbk} G$ is.  In fact, every finitely generated Frobenius symmetric tensor category of moderate growth over the complex numbers is equivalent to the category of finite-dimensional representations of some quasireductive supergroup.  This follows from Deligne's theorem \cite{D} and Prop.~2.3.1 of \cite{ES}.

	In \cite{SSh}, the first two named authors introduced the notion of a \emph{splitting subgroup} $\KK\sub\GG$, which analogizes the condition of containing a Sylow subgroup.  Namely, we require that $\KK$ is quasireductive and that the restriction functor gives a faithful embedding of $\DD^b(\Rep\GG)$ inside $\DD^b(\Rep\KK)$. 
	From this one may define a \emph{Sylow subgroup} of $\GG$ to be a minimal splitting subgroup of $\GG$. With this definition, the existence of Sylow subgroups is a trivial matter.
	
	In \cite{SSh} and \cite{SV}, nontrivial splitting subgroups were constructed for simple supergroups by computing volumes of certain homogeneous superspaces.  Already these results led to strong projectivity criterion for simple supergroups, as explained in \cite{SV}. However, the questions of whether these subgroups are Sylow, i.e.~minimal, and further if they are unique up to conjugacy, were not addressed.
	
	In this paper we deal with these questions for arbitrary quasireductive supergroups.  Namely, we prove the following:
	
	\begin{thm}\label{thm intro 1}
		Let $\GG$ be quasireductive.  Then all Sylow subgroups of $\GG$ are unique up to conjugacy.
	\end{thm}
	
	To prove Theorem \ref{thm intro 1}, we introduce analogs of $p$-groups to the super setting, which we call \emph{0-groups}.  Just as in the case of finite groups, one can view $0$-groups as an orthogonal notion to that of a reductive supergroup $\GG$, i.e.~one for which $\Rep\GG$ is semisimple. 
	
	We use 0 in our notation on the one hand because we are working in characteristic $0$.  Another reason is as follows: one may define a $p$-group $P$ to be one for which $\vol(P/H)=|P/H|$ is divisible by $p$ for every nontrivial subgroup $H\sub P$.  Equivalently, $\vol(P/H)$ is zero in $\Bbbk$.  For $0$-groups, we have the following equivalent characterizations. (We refer to Section \ref{section char 0 gps} for definitions.)
	
	\begin{thm}[Theorem/Definition]\label{thm intro 2} Let $\OO$ be quasireductive.  Then $\OO$ is a $0$-group if one of any of the following equivalent conditions hold:  
		\begin{enumerate}
			\item $\OO$ is its own Sylow subgroup.
			\item $\vol(\OO/\KK)=0$ for any proper, quasireductive subgroup $\KK\sub\OO$ with $\Ber(\o/\k)$ a trivial $\KK$-module (a necessary condition in order for volume to be defined).
			\item $\OO$ is oddly generated and $\o_{\ol{1}}^{neat}=\{0\}$.
			\item $\OO$ is oddly generated and the nil-cone is contained in the cone of self-commuting odd elements. 
		\end{enumerate}
	\end{thm}
	In Theorem \ref{thm intro 2}, we write $\vol(\OO/\KK)$ for the volume of an associated compact CS manifold, see Section 3.  In the super setting, volume forms are generalized to Berezin forms, hence the importance of the Berezin module $\Ber(\o/\k)$ in defining volumes. 
	Understanding the relationship between volumes and splitting subgroups was the key ingredient that allowed us to make progress on these problems.  Namely, we have:
	
	\begin{thm}\label{thm intro 3}
		Let $\KK\sub\GG$ be quasireductive supergroups.  Then $\KK$ is splitting in $\GG$ if and only if $\Ber(\g/\k)$ is a trivial $\KK$-module, and $\vol(\GG/\KK)\neq0$.
	\end{thm}
	
	Using Theorem \ref{thm intro 3}, we obtain the following characterization of Sylow subgroups.
	
	\begin{thm}\label{thm intro 4}
		Let $\OO\sub\GG$ be quasireductive supergroups.  Then $\OO$ is a Sylow subgroup if and only if it is a splitting $0$-subgroup.
	\end{thm}
	
	While Theorem \ref{thm intro 2} is a very neat characterization of $0$-groups, it is useful in practice to have a more explicit description of what a $0$-group looks like.  In the setting of finite groups, one could never hope to classify all $p$-groups; however they are all nilpotent, and this controls aspects of their representation theory.
	
	While the problem of classifying $0$-groups is also wild, we have the following:
	
	\begin{thm}\label{thm intro 5}
		Every $0$-group $\OO$ is isomorphic to a central extension of $\TT\times\VV$, where $\TT$ is a Takiff 0-supergroup, and $\VV$ is an odd abelian supergroup.
	\end{thm}
	
	For the definition of a Takiff 0-supergroup, we refer to Section \ref{section takiff}.   Theorem \ref{thm intro 5} allows us to understand something about the representation theory of $0$-groups, which is both very useful in the proof of Theorem \ref{thm intro 1}, and will be critical in future works.
	
	\subsection{Maximal 0-subgroups}  Our analogy with finite groups breaks down in the following way: for a finite group $G$, every maximal $p$-subgroup is Sylow.  However it is not true that every maximal $0$-subgroup is Sylow, i.e.~it need not be splitting.  A simple counterexample is provided by the subgroup of $\GG\LL(n|n)$ corresponding to the following subalgebra of $\g\l(n|n)$:
	\[
	\begin{bmatrix}
		A & B\\ \lambda I_n & A
	\end{bmatrix},
	\]
	where $A,B$ are arbitrary $n\times n$ matrices, and $\lambda\in\C$.
	
	We expect that there are only finitely many maximal $0$-subgroups up to conjugacy.  An important question is whether they have a representation-theoretic meaning, and at this point we have no indications as such.  The study of maximal $0$-subgroups will be taken up in future work.
	
	\subsection{The third Sylow theorem}
	
	The third Sylow theorem may be stated as follows: if $G$ is a finite group and $P$ is a Sylow subgroup of $G$, then $|G/N_G(P)|$ is congruent to $1$ (mod $p$), i.e. $\vol(G/N_G(P))=1$ in the base field.  
	
	Now let $\GG$ be quasireductive with Sylow subgroup $\OO$.  Our definition of volume is only well-defined up to non-zero scaling, so we cannot hope to assign meaning to $\vol(\GG/\NN_\GG(\OO))$.  However we have another approach to understand the `size' of $\GG/\NN_\GG(\OO)$. Let $\g=\operatorname{Lie}\GG$, and let $x\in\g_{\ol{1}}$ such that $[x,x]$ is semisimple in $\g_{\ol{0}}$; we call such an $x$ homological.  Then we obtain the Duflo-Serganova functor $DS_x:\Rep\GG\to s\operatorname{Vec}$, a symmetric monoidal functor (see Section \ref{section third sylow thm}).  For a finite-dimensional module $V\in\Rep\GG$, we have $\operatorname{sdim}(DS_xV)=\operatorname{sdim} V$, so in this sense $DS_x$ preserves size.  
	
	\begin{thm}
		Let $\TT\sub\OO_0$ be a maximal torus of $\OO_0$, and define
		\[
		W_{\GG}:=\NN_{\GG}(\TT)/\CC_{\GG}(\TT), \ \ \ W_{\NN_{\GG}(\OO)}:=\NN_{\NN_{\GG}(\OO)}(\TT)/\CC_{\NN_{\GG}(\OO)}(\TT).
		\]
		Then $W_{\NN_{\GG}(\OO)}\sub W_{\GG}$ are finite groups, and for a generic (see Section \ref{section third sylow thm}) homological $x\in\g_{\ol{1}}$, we have a natural isomorphism of algebras:
		\[
		DS_x\C[\GG/\NN_{\GG}(\OO)]\cong\C[W_{\GG}/W_{\NN_{\GG}(\OO)}].
		\]
	\end{thm}
	
	In this way, we may view $\GG/\NN_{\GG}(\OO)$ as having `size' $
	|W_{\GG}/W_{\NN_{\GG}(\OO)}|$, which we note is a finite number, and is often equal to one (see Section \ref{section table}).

	\subsection{Supergroups vs.~superalgebras}  So far we have only discussed supergroups, but in fact our above discussion works just as well for quasireductive Lie superalgebras (see Definition \ref{defn qred algebra}).  In fact, $0$-superalgebras and Sylow subalgebra are often easier to discuss and work with than $0$-groups and Sylow subgroups. 
	
	In Section 2 we will explain the passages between supergroups and superalgebras, and throughout the paper we will work with one or the other, as is convenient.  All structure theorems stated about supergroups, for example, will have an analogue for superalgebras, and vice-versa.
	
	\subsection{Future work and open problems} In an upcoming preprint, the first two named authors and J.~Pevtsova will continue the work of extending the analogy between finite groups and supergroups.  Namely, we will define analogs of elementary abelian subgroups in the super setting, and prove a strong projectivity criteria for supergroups which analogizes Chouinard's theorem.  Using this, we will give an explicit description of the cohomological support variety.
	
	Further questions of future interest are to develop a theory of defect subgroups for blocks, and understanding what Green's correspondence looks like for supergroups.  We expect that a meaningful, general definition of atypicality and defect should arise from our machinery. 
	
	\subsection{Outline of article} In Section 2 we recall the definitions of quasireductive supergroups/Lie superalgebras, and explain how to pass between them via the notion of global forms.  The notion of an oddly generated supergroup/superalgebra will be of particular importance later on.  In Section 3 we define algebraic integrals on homogeneous superspaces and explain the consequences of admitting an algebraic integral.  Section 4 recalls necessary language and facts from the theory of Berezin integration on CS manifolds.  In Section 5 we apply Berezin integration to CS manifolds arising from complex homogeneous superspaces, and use this to classify when a homogeneous superspace admits an algebraic integral.  Section 6 then defines Sylow subgroups and 0-groups, and relates their definitions to the work on volumes in Section 5.  Section 7 characterizes and classifies $0$-groups, and deduces facts about their representation theory.  In Section 8 we prove that all Sylow subgroups are unique up to conjugacy.  Finally, in Section 9 we prove the third Sylow theorem as stated in the introduction.
	
	\subsection{Acknowledgements} The authors would like to thank Inna Entova-Aizenbud and Julia Pevtsova for many helpful discussions around this project.  We further thank Chris Hone and Geordie Williamson for patiently explaining to us aspects of modular representation theory.  The first author was partially supported by NSF grant 2001191 and by Tromso Research Foundation (project “Pure Mathematics in Norway"). The first author is  also thankful to the Sydney Mathematical Research Institute for support and hospitality during her visit.  The second author was partially supported by ARC grant DP210100251.
	
	\section{Quasireductive supergroups and quasireductive Lie superalgebras}
	
	Throughout, unless stated otherwise, we work over the complex numbers $\C$.  For a super vector space $V$, we write $V=V_{\ol{0}}\oplus V_{\ol{1}}$ for its parity decomposition. 
	
	\subsection{Supergroups and superalgebras} By a supergroup $\GG$ we mean an algebraic supergroup, that is the spectrum of a supercommutative Hopf superalgebra, which we write as $\C[\GG]$.  Given a supergroup $\GG$, we write $\GG_0$ for its even part, given by the spectrum of $\C[\GG]/(\C[\GG]_{\ol{1}})$, which will be an algebraic group.  
	We use the symbols $\GG,\HH,\KK,\OO\dots$ for supergroups, and the symbols $\g,\h,\k,\o\dots$ for Lie superalgebras.  If not otherwise stated, $\g,\h,\k,\o,\dots$ will be the Lie superalgebra of $\GG,\HH,\KK,\OO,\dots$.  However in some cases our notation will be set up so that $\g$ and $\operatorname{Lie}\GG$ are distinct.  However, usually $\g\sub\operatorname{Lie}\GG$. We hope the relationship between $\GG$ and $\g$ is always clear from context.
	
	Whenever we discuss subsupergroups of a supergroup $\GG$ we will simply refer to them as subgroups, and same for Lie subalgebras, etc. 
	
	\begin{definition}
		We say that a supergroup $\GG$ is quasireductive if the connected component of the identity of $\GG_0$ is reductive.  In particular, we do not assume that $\GG$ is connected.  We write $\Rep\GG$ for the abelian, symmetric monoidal category of rational $\GG$-modules; that is,  the category of comodules over $\C[\GG]$.  We emphasize that infinite-dimensional modules are included.
	\end{definition}
	
	In much of the rest of the text we use many known facts about quasireductive supergroups.  We refer to \cite{S} for the foundations of their structure and representation theory.
	
	\subsection{Global forms of quasireductive Lie superalgebras}  
	
	\begin{definition}\label{defn qred algebra}
		We say that a Lie superalgebra $\g$ is quasireductive if $\g_{\ol{0}}$ is reductive and acts $\ad$-semisimply on $\g$. We write $\Rep_{\g_{\ol{0}}}\g$ for the category of locally finite $\g$-modules which are semisimple  over $\g_{\ol{0}}$.  We say a module is locally finite if it is a union of its finite-dimensional submodules.

	\end{definition}
	
	Note that there  may not exist a quasireductive supergroup $\GG$ for which $\g=\operatorname{Lie}\GG$.
	
	\begin{definition}
		We say that $\k\sub\g$ are quasireductive superalgebras if $\g$ is quasireductive and $\k$ is a subalgebra of $\g$ such that $\k_{\ol{0}}$ acts $\text{ad}$-semisimply on $\g$.  In particular, $\k$ will also be quasireductive under this definition.
	\end{definition}
	
	\begin{definition}
		Let $\g$ be a quasireductive Lie superalgebra (see Definition \ref{defn qred algebra}).  Then we call a quasireductive supergroup $\GG$ a global form of $\g$ if the following conditions hold:
		\begin{enumerate}
			\item $\g$ is a Lie subalgebra of $\operatorname{Lie}\GG$;
			\item if $\KK\sub\GG$ is any subgroup with $\g\sub \operatorname{Lie}\KK$, then $\KK=\GG$.
		\end{enumerate}
	\end{definition}
	
	\begin{example}
		Let $\g$ be the quasireductive Lie superalgebra with even basis $x$ and odd basis $H$, and such that $[H,H]=x$.  Then for any positive integer $n$ there exits a global form $\GG$ of $\g$ of dimension $(n|1)$.  In fact for a fixed $n>1$, there are uncountably many such global forms up to isomorphism.
	\end{example}
	
	\begin{lemma}\label{lemma global forms properties}
		\begin{enumerate}
			\item  Every quasireductive Lie superalgebra admits a global form, and it is always connected.
			\item If $\GG$ is a global form of $\g$, then $\g_{\ol{1}}=(\operatorname{Lie}\GG)_{\ol{1}}$ and $\GG_0$ is a global form of the Lie algebra $\g_{\ol{0}}$.
			\item If $\k\sub\g$ are quasireductive Lie superalgebras and $\GG$ is a global form of $\g$, then there exists a unique subgroup $\KK\sub\GG$ such that $\k\sub\operatorname{Lie}\KK$, and $\KK$ is a global form of $\k$.
		\end{enumerate}
	\end{lemma}
	
	\begin{proof}
		We first prove (3).  Suppose that $\GG$ is a global form of $\g$.  Then let $\KK\sub\GG$ be the subgroup of $\GG$ corresponding to the super Harish-Chandra pair (see \cite{CF}) $(\KK_0,\k)$, where $\KK_0$ is the minimal algebraic subgroup of $\GG_0$ whose Lie superalgebra contains $\k_{\ol{0}}$.  Then it is clear that $\KK$ is a global form of $\k$ and is connected.
		
		To prove (1), apply Ado's theorem for Lie superalgebras to obtain a faithful representation $\k\sub\g\l(V)$ which is semisimple over $\k_{\ol{0}}$.  Now we are reduced to (3) where $\GG=\GG\LL(V)$.
		
		Finally, to prove (2), if $\GG$ is a global form of $\g$, then use (3) to obtain $\KK_0\sub\GG_0$ which is a global form of $\g_{\ol{0}}$.  Note that $\g\sub\operatorname{Lie}\GG$ is $\KK_0$-stable with respect to the adjoint action and $[\g_{\ol{1}},\g_{\ol{1}}]\subset \g_{\ol{0}}\subset\operatorname{Lie}\KK_0$.  Hence we obtain a super Harish-Chandra pair $(\KK_0,\operatorname{Lie}\KK_0\oplus\g_{\ol{1}})$, which will give a quasireductive subgroup of $\GG$ whose Lie superalgebra contains $\g$. Thus $\KK=\GG$ by the definition of a global form.
	\end{proof}
	
	\begin{lemma}\label{lemma full subcat}
		Let $\g$ be quasireductive, and let $\GG$ be a global form of $\g$.  Then $\Rep\GG$ is a full monoidal, topologizing, Serre subcategory of $\Rep_{\g_{\ol{0}}}\g$, and in particular contains the principal block of $\Rep_{\g_{\ol{0}}}\g$.
	\end{lemma}
	
	\begin{proof}
		We have a natural faithful, exact tensor functor $F:\Rep^{fd}\GG\to\Rep^{fd}_{\g_{\ol{0}}}\g$, where $(-)^{fd}$ denotes the full subcategories of finite-dimensional modules.  Note that these are super-Tannakian categories.  Let $\CC$ be the finitely generated super-Tannakian subcategory of $\Rep^{fd}_{\g_{\ol{0}}}\g$ generated by $F(V)$, where $V$ is a finite-dimensional $\GG$-module.  By super-Tannakian reconstruction, $\CC=\Rep^{fd}\GG'$ for some supergroup $\GG'$, and we have a natural inclusion of supergroups $\GG'\sub\GG$.  However, we have $\g\sub\operatorname{Lie}\GG'\sub\operatorname{Lie}\GG$.  Since $\GG$ is a global form of $\g$, this implies $\GG'=\GG$, meaning that $F$ is also full.
	\end{proof}
	
	\subsection{Oddly generated supergroups and superalgebras}
	
	\begin{definition}\label{defn oddly generated}    
		\begin{enumerate}
			\item We say that a quasireductive supergroup $\GG$ is oddly generated if $\GG_0$ is a global form of $[(\operatorname{Lie}\GG)_{\ol{1}},(\operatorname{Lie}\GG)_{\ol{1}}]$. 
			\item We say that a quasireductive Lie superalgebra $\g$ is oddly generated if $\g_{\ol{0}}=[\g_{\ol{1}},\g_{\ol{1}}]$.
		\end{enumerate}
		
		\begin{remark}
			\begin{enumerate}
				\item[(i)] An oddly generated supergroup is necessarily connected.
				\item[(ii)] If $\g$ is an oddly generated superalgebra, then any global form of it is also oddly generated.  However, if $\GG$ is oddly generated then $\operatorname{Lie}\GG$ need not be oddly generated.
			\end{enumerate}
		\end{remark}
		
	\end{definition}
	
	To help put the following lemma in context, we remark that given a global form $\GG$ of $\g$ and a subgroup $\HH\sub\GG$, there may or may not exist a subalgebra $\h\sub\g$ for which $\HH$ is a global form of $\h$.  Further, when such a subalgebra exists it need not be unique.
	\begin{lemma}\label{lemma bijection oddly generated}
		Suppose that $\g$ is quasireductive, and let $\GG$ be a global form of $\g$.  Then we have a $\GG_0$-equivariant bijection of $\GG_0$-sets
		\[
		\mathfrak{G}:\text{sAlg}_{ogen}\to \text{sGrp}_{ogen}.
		\]
		Here $\text{sAlg}_{ogen}$ is the set of oddly generated subalgebras $\h\sub\g$, and $\text{sGrp}_{ogen}$ is the set of oddly generated subgroups $\HH\sub\GG$.
	\end{lemma}
	\begin{proof}
		The map $\mathfrak{G}$ takes an oddly generated subalgebra $\h\sub\g$ to a global form $\mathfrak{G}(\h)\sub\GG$ as in (3) of Lemma \ref{lemma global forms properties}.  The inverse map $\mathfrak{G}^{-1}$ takes an oddly generated subgroup $\HH\sub\GG$ with the given properties to the Lie superalgebra 
		\[
		\mathfrak{G}^{-1}(\HH)=[(\operatorname{Lie}\HH)_{\ol{1}},(\operatorname{Lie}\HH)_{\ol{1}}]+(\operatorname{Lie}\HH)_{\ol{1}}.
		\]
		This correspondence is both clearly bijective and intertwines the adjoint action of $\GG_0$ on $\text{sAlg}_{ogen}$ with the conjugation action on $\text{sGrp}_{ogen}$.
	\end{proof}
	
	\section{Algebraic integrals on homogeneous affine supervarieties}
	
	\subsection{Homogeneous spaces}  Given a subgroup $\HH\sub\GG$, we may construct the smooth homogeneous supervariety $\GG/\HH$, see \cite{MT}.  Note that $\GG/\HH$ admits an action of $\GG$ on the left by translation.  If both $\GG$ and $\HH$ are quasireductive, then $\GG/\HH$ will be an affine supervariety. 
	
	\subsection{Geometric induction}\label{section induction functors} 
	Given a a subgroup $\HH\sub\GG$, we have a functor \linebreak $\Ind_{\HH}^{\GG}:\Rep\HH\to\Rep\GG$ given by 
	\[
	\Ind_{\HH}^{\GG}V=(V\otimes\C[\GG])^{\HH},
	\]
	where we take the $\HH$ action on $\C[\GG]$ by right translation. 
	In the case that both $\HH$ and $\GG$ are quasireductive, $\Ind_{\HH}^{\GG}(-)$ is exact.  
	
	Observe that $\Ind_{\HH}^{\GG}\C=\C[\GG]^{\HH}=\C[\GG/\HH]$.  Further, $\Ind_{\HH}^{\GG}V$ always has the structure of a $\C[\GG/\HH]$-module which is compatible with the $\GG$-action.
	
	\begin{lemma}\label{lemma simple module}  Let $\HH\sub\GG$ be quasireductive supergroups, and let $V$ be a simple $\HH$-module. 
		Then $\Ind_{\HH}^{\GG}V$ is simple as a $\C[\GG/\HH]-\GG$ module.  In particular, $\C[\GG/\HH]$ is a simple $\GG$-algebra.  
	\end{lemma}
	\begin{proof}
		This follows immediately from the fact that $\Ind_{\HH}^{\GG}$ is an exact functor and Barr-Beck monadicity.
		
	\end{proof}
	
	\subsection{Algebraic integrals}
	\begin{definition}
		Let $\HH\sub\GG$ be quasireductive supergroups.  An algebraic integral on $\GG/\HH$ is a $\GG$-equivariant map $\iota:\C[\GG/\HH]\to\C$, which may be either even or odd.  
	\end{definition} 
	Given an algebraic integral $\iota$ on $\GG/\HH$, we obtain a symmetric bilinear form $(-,-)_{\iota}$ on $\C[\GG/\HH]$ given by $(f,g)_{\iota}=\iota(fg)$, and satisfying $(fg,h)_{\iota}=(f,gh)_{\iota}$.
	
	\begin{lemma}\label{lemma nondeg form}
		If $\iota\neq0$ then $(-,-)_{\iota}$ is nondegenerate, i.e.~its kernel is trivial as a bilinear form.
	\end{lemma}
	
	\begin{proof}
		Suppose that $(-,-)_{\iota}$ has a nontrivial kernel, which we call $K$.  Clearly $K$ is $\GG$-stable, and if $g\in\C[\GG/\HH]$, $f\in K$, then $(fg,h)_{\iota}=(f,gh)_{\iota}=0$. Thus $fg\in K$, meaning $K$ is a $\GG$-stable ideal.  This contradicts Corollary \ref{lemma simple module}, so we are done.
	\end{proof}
	
	Notice that the form $(-,-)_{\iota}$ induces a $\GG$-equivariant map $\Phi_{\iota}:\C[\GG/\HH]\to\C[\GG/\HH]^*$.
	
	\begin{cor}\label{cor iso from form}
		If $\iota\neq0$ then the induced map $\Phi_{\iota}:\C[\GG/\HH]\to\C[\GG/\HH]^*$ is an isomorphism onto the $\GG$-finite vectors $\Gamma_\GG(\C[\GG/\HH]^*)$ of $\C[\GG/\HH]^*$.
	\end{cor}
	
	\begin{proof}
		By Lemma \ref{lemma nondeg form}, $\Phi_{\iota}$ is injective. For surjectivity, first observe that for any simple $\GG_0$-module $L$, we have 
		\[
		[\operatorname {Res}_{\GG_0}^{\GG}\C[\GG/\HH]:L]<\infty.
		\]
		Indeed, we have:
		\begin{eqnarray*}
			\Hom_{\GG_0}(L,\operatorname{Res}_{\HH}^{\GG}\C[\GG/\HH])& \cong & \Hom_{\GG}(\UU\g\otimes_{\UU\g_{\ol{0}}}L,\Ind_{\HH}^\GG\C)\\
			& \cong &\Hom_{\HH}(\operatorname{Res}_{\HH}^{\GG}(\UU\g\otimes_{\UU\g_{\ol{0}}}L),\C),    
		\end{eqnarray*}
		and the latter is clearly finite.		Thus we may write
		\[
		\C[\GG/\HH]\cong\bigoplus\limits_{L}L^{\oplus n_{L}},
		\]
		where the isomorphism is of $\GG_0$-modules, and $L$ runs over all simple $\GG_0$-modules up to parity.  It follows that 
		\[
		\C[\GG/\HH]^*\cong \prod\limits_{L}(L^*)^{\oplus n_{L}},
		\]
		It is straightforward to see that
		\[
		\Gamma_\GG(\prod\limits_{L}(L^*)^{\oplus n_{L}})=\bigoplus\limits_{L}(L^*)^{\oplus n_{L}}.
		\]
		It follows that $\Phi_{\iota}$ defines an injective map
		\[
		\C[\GG/\HH]\cong \bigoplus\limits_{L}L^{\oplus n_{L}}\hookrightarrow \bigoplus\limits_{L}(L^*)^{\oplus n_{L}}.
		\]
		Thus we obtain $n_{L^*}\leq n_L$ for all $L$, implying that $n_L=n_{L^*}$ for all $L$, meaning that $\Phi_{\iota}$ is an isomorphism.
	\end{proof}

	\begin{lemma}\label{lemma integral unique}
		Suppose that $\GG/\HH$ admits a nonzero algebraic integral $\iota$.  Then any other algebraic integral on $\GG/\HH$ is a scalar multiple of $\iota$.
	\end{lemma}
	
	\begin{proof}
		Indeed, we see that $\Phi_{\iota}(1)=\iota$, thus the image of $1$ determines the integral.  However the existence of $\iota\neq0$ implies that  $\Gamma_\GG(\C[\GG/\HH]^*)\cong\C[\GG/\HH]$ by Corollary \ref{cor iso from form}.  Thus 
		\[
		(\C[\GG/\HH]^*)^{\GG}=\Gamma_\GG(\C[\GG/\HH]^*)^\GG\cong \C[\GG/\HH]^{\GG}=\C\langle1\rangle.
		\]
	\end{proof}
	
	\section{CS manifolds, Berezin forms, and Berezin integration}
	
	For further details on CS manifolds, Berezins, and Berezin integration, we refer to \cite{Manin}, \cite{R}, and \cite{W}.  
	
	\subsection{Integration on CS manifolds} Recall that a CS supermanifold is a locally ringed space that is locally isomorphic to 
	\[
	\R_{CS}^{m|n}:=(\R^m,C(\R^m)\otimes{\bigwedge}^\bullet(\C^{n})),
	\]
	for some $m,n\in\N$.  Here $C(\R^m)$ denotes the algebra of smooth, complex-valued functions on $\R^m$.  
	
	Let $\mathcal{X}$ be either a CS supermanifold or a smooth complex algebraic supervariety.  Then we have the usual notion of the cotangent bundle $\Omega_\XX$ on $\XX$.  Thus we may consider $\Ber_\XX:=\Ber(\Omega_\XX)$, the sheaf of Berezin forms on $\XX$. To be precise, for a vector bundle $\mathcal{V}$ of dimension $(r|s)$, $\Ber(\mathcal{V})$ is defined to be the line bundle with local generator given by a local, ordered, homogeneous basis $(a_1,\dots,a_r,b_1,\dots,b_s)$ of $\mathcal{V}$, and such that $(a_1,\dots,a_r,b_1,\dots,b_s)=\Ber(A)(c_1,\dots,c_r,d_1,\dots,d_s)$ if the second basis is transformed into the first basis via $A$, a local automorphism of $\mathcal{V}$. 
	
	We now restrict to the CS manifold setting to define Berezin integration.  Given homogeneous coordinates $x_1,\dots,x_m,\xi_1,\dots,\xi_n$ on $\R^{m|n}_{CS}$, we write $\frac{dx_1\cdots dx_m}{d\xi_1\cdots d\xi_n}$ for the generator of $\Ber(\Omega_{\R^{m|n}_{CS}})$ given by $(dx_1,\dots,dx_m,d\xi_1,\dots,d\xi_n)$.
	
	Given a compactly supported Berezin form $\omega$ on $\R^{m|n}_{CS}$ with coordinates $x_1,\dots,x_m,\xi_1,\dots,\xi_n$, we may present it as
	\[
	\omega=\sum\limits_{I}f_I\xi_I\frac{dx_1\cdots dx_m}{d\xi_1\cdots d\xi_n},       
	\]
	where the $f_I\in C(\R^m)$ are compactly supported, and $I$ runs over all subsets of $\{1,\dots,n\}$.  Then we define 
	\[
	\int_{\R^{m|n}_{CS}}\omega=\int_{\R^m}f_{\{1,\dots,n\}}dx_1\cdots dx_m.
	\]
	That this expression is independent of coordinate changes is proven in Theorem 11.3.2 of \cite{R}.
	
	Now let $\omega\in\Ber_\XX$ be an arbitrary compactly supported section, and let $(\UU_i,\varphi_i)$ be a partition of unity of $\XX$.  We define
	\[
	\int_{\XX}\omega=\sum\limits_i\int_{\UU_i}\varphi_i\omega.
	\]
	By the independence of the choice of coordinates on $\R^{m|n}_{CS}$ in the local setting, we find this is again well-defined.
	
	\subsection{Integration along  odd fibers} We define an integration along odd fibres morphism, which we emphasize works in both the algebraic and CS settings.  At the end we show they are compatible in the setting of interest to us.
	
	Let $\XX$ be a smooth supervariety \textbf{or} a CS manifold, and suppose that we have a splitting $\pi:\XX\to \XX_0$.  Define the integration along odd fibres map 
	\[
	\pi_*:\Ber_\XX\to\Omega_{\XX_0}^{top}
	\]
	to be the following map of sheaves: first suppose that $\XX=\Spec A$ is affine if $\XX$ is a supervariety, and if $\XX$ is a CS manifold then suppose that $\XX=\R^{m|n}_{CS}$, and set $A=\Gamma(\XX,\OO_\XX)$.  In both cases write $\ol{A}=A/(A_{\ol{1}})$.
	
	Our splitting $\pi$ corresponds to a splitting $\pi^*:\ol{A}\to A$ of the projection $A\to\ol{A}$.   Let us further assume that $A$ has global coordinates by shrinking down our affine neighborhood further.  If $x_1,\dots,x_m$ are global coordinates for $\ol{A}\sub A$, and $\xi_1,\dots,\xi_n$ are global odd coordinates for $A$, then $x_1,\dots,x_m,\xi_1,\dots,\xi_n$ will be global coordinates for $A$, and we have an algebra isomorphism 
	\[
	A\cong \ol{A}[\xi_1,\dots,\xi_n].
	\]
	Since $\frac{dx_1\cdots dx_m}{d\xi_1\cdots d\xi_n}$ trivializes $\Ber_\XX$, we have an isomorphism of $\ol{A}$-modules:
	\[
	\Ber_\XX\cong \bigoplus\limits_{I}\ol{A}\xi_I\frac{dx_1\cdots dx_m}{d\xi_1\cdots d\xi_n},
	\] 
	where $I$ runs over all subsets of $\{1,\dots,n\}$.  It therefore makes sense to define
	\[
	\pi_*\left(\sum\limits_{I}f_I\xi_I \frac{dx_1\cdots dx_m}{d\xi_1\cdots d\xi_n}\right)=f_{\{1,\dots,n\}}dx_1\cdots dx_m.
	\]
	That this is independent of coordinates is given for instance in Section 6 of Chapter 4 of \cite{Manin}.  To extend to all of $\XX$, we patch together the local constructions.  The following lemmas are now clear.
	
	\begin{lemma}\label{lemma surj int fibres}
		Let $\XX$ be either a smooth algebraic supervariety or a CS manifold.  If we have a splitting $\pi:\XX\to \XX_0$, then $\pi_*:\Ber_{\XX}\to\Omega_{\XX_0}^{top}$ is a surjective map of sheaves.  In particular, if $\XX_0$ is an affine variety or a manifold, then it defines a surjective map on global sections.  
	\end{lemma}

	\begin{lemma}\label{lemma integral fibres}
		If $\XX$ is a CS manifold with a splitting $\pi:\XX\to \XX_0$, then for a compactly supported $\omega\in\Gamma(\XX,\Ber_\XX)$ we have
		\[
		\int_{\XX}\omega=\int_{\XX_0}\pi_*(\omega).
		\]
	\end{lemma}
	\begin{proof}
		This follows directly from the definition of Berezin integration and $\pi_*$.
	\end{proof}

	\subsection{Real forms} 
	Suppose that $\XX$ is a smooth algebraic supervariety such that $\XX_0$ has a real form $(\XX_0)_{\R}\sub \XX_0$.  Then by \cite{V}, $\XX_{\R}:=((\XX_0)_{\R},\OO_{\XX}^{an}|_{(\XX_0)_{\R}})$ has the structure of a CS manifold.  Write $i:\XX_{\R}\to \XX$ for the natural map of locally ringed spaces, and let $i_0:(\XX_0)_{\R}\to \XX_0$ be the embedding of underlying spaces.
	
	Then we have a natural map $i^*:\Ber_\XX\to\Ber_{\XX_\R}$, where to emphasize, $\Ber_{\XX}$ denotes the Berezin sheaf of the complex algebraic cotangent sheaf $\Omega_\XX$, and $\Ber_{\XX_{\R}}$ is the Berezin sheaf of the CS cotangent sheaf $\Omega_{\XX_{\R}}$.  
	
	\begin{lemma}\label{lemma int fibres commutes with real form}
		A splitting of algebraic supervarieties $\pi:\XX\to \XX_0$ determines a  unique, compatible splitting $\pi_{\R}:\XX_{\R}\to(\XX_0)_{\R}$, and we have for $\omega\in\Ber_\XX$:
		\[
		i_0^*(\pi_*(\omega))=(\pi_{\R})_{*}(i^*(\omega)).
		\]
	\end{lemma}
	
	\begin{proof}
		The fact that $\pi$ determines a unique splitting $\pi_{\R}$ is immediate.  To obtain the equality, we note that if $x_1,\dots,x_m,\xi_1,\dots,\xi_n$ are local coordinates on $\XX$, then they restrict to local coordinates on $\XX_{\R}$.  The rest follows from the formula for integration along odd fibers.
	\end{proof}
	
	\section{Volumes of homogeneous superspaces}
	\subsection{Existence of compact forms}  In the following, we recall that maximal compact subgroups of a complex algebraic group whose identity component is reductive are unique up to conjugacy, see \cite{Borelcompact}.
	
	Let $\GG$ be a quasireductive supergroup with quasireductive subgroup $\HH$.  Let $\HH_{0,c}$ denote a maximal compact subgroup of $\HH_0$, and let $\GG_{0,c}$ be a maximal compact subgroup of $\GG_0$ containing $\HH_{0,c}$.  These give compact real forms of $\GG_0,\HH_0$ respectively, and so we obtain compact real CS forms $\HH_c\sub \GG_c$ of $\HH$ and $\GG$, respectively.
	
	Further, $\GG_{0,c}/\HH_{0,c}\sub\GG_0/\HH_0$ is a real form, and so we may consider the compact CS manifold given by 
	\[ 
	\GG_c/\HH_c:=(\GG_{0,c}/\HH_{0,c},\OO_{\GG/\HH}^{an}|_{\GG_{0,c}/\HH_{0,c}}).
	\]
	
	\begin{lemma}\label{lemma berezin form trivializes}
		Let $\HH\sub\GG$ be quasireductive supergroups with Lie superalgebras $\h\sub\g$.  Then we have a natural isomorphism of super vector spaces
		\[
		\Ber(\g/\h)^{\HH}\xto{\sim}\Gamma(\GG/\HH,\Ber_{\GG/\HH})^{\GG}.
		\]
		In particular, $\dim\Gamma(\GG/\HH,\Ber_{\GG/\HH})^{\GG}\leq 1$, and further if $\omega\in\Gamma(\GG/\HH,\Ber_{\GG/\HH})^{\GG}$ is non-zero, then it trivializes $\Ber_{\GG/\HH}$.
	\end{lemma}
	
	\begin{proof}
		Indeed, 
		\[
		\Gamma(\GG/\HH,\Ber_{\GG/\HH})^{\GG}={^\GG}(\C[\GG]\otimes\Ber(\g/\h))^{\HH}=\C\langle1\rangle\otimes\Ber(\g/\h)^{\HH}.
		\]
	\end{proof} 
	\begin{prop}\label{prop integral nonzero}
		Let $\HH\sub\GG$ be quasireductive  supergroups, set $\XX:=\GG/\HH$, and suppose that $\omega_{\XX}\in\Gamma(\XX,\Ber_{\XX})$ is non-zero and $\GG$-invariant.  Let $\XX_c=\GG_c/\HH_c$ with inclusion $i:\XX_c\hookrightarrow\XX$. 
		Then the morphism $\iota_{\omega_\XX}:\Gamma(X,\OO_\XX)\to\C$ given by
		\[
		\iota_{\omega_{\XX}}: f\mapsto \int_{\XX_c}i^*(f\omega_\XX),
		\]
		is non-zero and $\GG$-equivariant.  In particular, $\iota_{\omega_{\XX}}$ defines an algebraic integral on $\GG/\HH$.
	\end{prop}
	\begin{proof}
		We first note that by Lemma \ref{lemma berezin form trivializes}, $\omega_{\XX}$ trivializes $\Ber_{\XX}$.  Further, since $\XX_0$ is the quotient of a reductive group by a reductive subgroup, $\Omega^{top}_{\XX_0}$ admits a $\GG_0$-invariant trivializing form.
		
		Since $\XX$ is smooth and affine, it is split; thus choose a splitting $\pi:\XX\to \XX_0$, and write $\pi_c:\XX_c\to(\XX_c)_0$ for the induced splitting of $\XX_c$ (see Lemma \ref{lemma int fibres commutes with real form}).  By Lemma \ref{lemma surj int fibres}, $\pi_*:\Ber_\XX\to\Omega^{top}_{\XX_0}$ is surjective, so there exists $f\in\Gamma(\XX,\OO_\XX)$ such that $\pi_*(f\omega_\XX)$ is $\GG_0$-invariant and trivializes $\Omega_{\XX_0}^{top}$. Thus if $i_0:(\XX_c)_0\hookrightarrow \XX_0$ is the natural inclusion, we have that $i_0^*(\pi_*(f\omega_\XX))$ is a trivializing top differential form on $(\XX_c)_0$, implying its integral over $(\XX_c)_0$ is nonzero. 
		Thus by Lemmas \ref{lemma integral fibres} and \ref{lemma int fibres commutes with real form},
		\begin{eqnarray*}
			\iota_{\omega_\XX}(f)& = &\int\limits_{\XX_c}i^*(f\omega_\XX)\\
			& = &\int\limits_{(\XX_c)_0}(\pi_c)_*(i^*(f\omega_\XX))\\
			& = &\int\limits_{(\XX_c)_0}i_0^*(\pi_*(f\omega_X))\neq0.
		\end{eqnarray*}
		We refer to Thm.~16 of \cite{SV} for the proof that $\iota_{\omega_\XX}$ is $\GG$-equivariant. 
	\end{proof}
	
	\subsection{Existence of integrals}  
	
	\begin{thm}\label{thm existence integral}
		Let $\HH\sub\GG$ be quasireductive.  Then there exists a nontrivial algebraic integral $\iota:\C[\GG/\HH]\to\C$ if and only if $\Ber(\g/\h)$ is trivial as an $\HH$-module (i.e.~is isomorphic to either $\C$ or $\Pi\C$.)  
	\end{thm}
	
	Note that Theorem \ref{thm existence integral} was proven in the analytic setting in \cite{AH}. 
	
	\begin{proof}
		Observe that $\Ber_{\GG/\HH}\cong\Ind_{\HH}^{\GG}\Ber(\g/\h)$.  We always have a $\GG$-equivariant natural pairing 
		\[
		\C[\GG/\HH]\otimes \Ber_{\GG/\HH}\to\C
		\]
		given by
		\[
		f\otimes \omega\mapsto \int\limits_{\GG_c/\HH_c}i^*(f\omega)
		\]
		as in Proposition \ref{prop integral nonzero}.  We claim this defines a perfect pairing.  To see this, notice that if it is not zero it is automatically perfect because otherwise a kernel would define nontrivial $\GG$-stable ideal in $\C[\GG/\HH]$ or a $\GG$-stable submodule of $\Ind_{\HH}^{\GG}\Ber(\g/\h)$, contradicting Lemma \ref{lemma simple module}.  On the other hand it is not zero by Lemmas \ref{lemma surj int fibres} and \ref{lemma integral fibres}, using the existence of a nonvanishing top form on $\GG_{0,c}/\HH_{0,c}$ whose integral is 0.
		
		Thus we obtain in every case an isomorphism $\Ind_{\HH}^{\GG}\Ber(\g/\h)\cong \Gamma_\GG\C[\GG/\HH]^*$ as in Corollary \ref{cor iso from form}. Therefore, if $\GG/\HH$ has an algebraic integral, we obtain via this isomorphism a $\GG$-invariant Berezin form, implying that $\Ber(\g/\h)$ is trivial. 
		Conversely, if $\Ber(\g/\h)$ is a trivial $\HH$-module, then a global $\GG$-invariant Berezin form exists, call it $\omega$. Then via Proposition \ref{prop integral nonzero} we obtain a non-zero algebraic integral $\iota_{\omega}$ as desired.
	\end{proof}
	
	\subsection{Volumes of homogeneous superspaces}
	
	\begin{definition}\label{definition volume}
		Suppose that $\XX$ is a compact cs manifold and let $\omega\in\Ber_{\XX}$ be a trivialization.  Then we define the volume of $\XX$ with respect to $\omega$ to be:
		\[
		\vol(\omega)=\vol_{\omega}(\XX):=\int\limits_{\XX}\omega.
		\]
		If $\HH\sub\GG$ are quasireductive supergroups, with $\Ber(\g/\h)$ a trivial $\HH$-module, we define
		\[
		\vol(\GG/\HH):=\vol(\omega_{\GG/\HH}|_{\GG_c/\HH_c}),   
		\]
		where $\omega_{\GG/\HH}$ is a non-zero $\GG$-invariant Berezin form on $\GG/\HH$, and we have chose compact forms $\HH_c\sub\GG_c$. Observe that $\vol(\GG/\HH)$ is only well-defined up to non-zero scalar; as a result it is only well-defined to state that $\vol(\GG/\HH)$ is either zero or non-zero.
	\end{definition}
	
	We remark that for $\vol(\GG/\HH)$ to be well-defined in Definition \ref{definition volume}, we need it to be independent of the choices $\HH_c\sub\GG_c$.  For this we use Cor.~1.3 of Chpt.~7 in \cite{Borelcompact}.

	For the following proposition, let $\pi:\EE\to\BB$ be a locally trivial fibration of compact CS manifolds with fibre $\FF$.  Then we have a short exact sequence of vector bundles on $\EE$:
	\[
	0\to \pi^*\Omega_{\BB}\to \Omega_{\EE}\to\Omega_{\EE/\BB}\to 0.
	\]
	This implies that $\Ber(\Omega_{\EE})\cong\Ber(\Omega_{\EE/\BB})\otimes\Ber(\pi^*\Omega_{\BB})$.
	
	\begin{prop}\label{prop integral bundle}
		Suppose that $\omega_\EE$, $\omega_\BB$, and $\omega_\FF$ are global Berezin forms on $\EE,\BB,$ and $\FF$ respectively.  Suppose further that $\Ber(\Omega_{\EE/\BB})$ is trivialized by a section $\omega_{\EE/\BB}$ satisfying that $\omega_{\EE/\BB}|_\FF=\omega_{\FF}$ along each fibre $\FF$, and that $\omega_\EE=\pi^*\omega_\BB\cdot \omega_{\EE/\BB}$.  Then,
		\[
		\vol(\omega_\EE)=\vol(\omega_\BB)\vol(\omega_\FF).
		\]
	\end{prop}
	
	\begin{proof}	
		Let $(\UU_i,\varphi)$ be a partition of unity of $\BB$ on which our fibre bundle is trivial, i.e.~$\pi^{-1}(\UU_i)\cong \UU_i\times \FF$.  We have 
		\begin{eqnarray*}
			\int_{\EE}\omega_\EE& = &\sum\limits_{i}\int_{\pi^{-1}(\UU_i)}\pi^*(\varphi_i)\pi^*(\omega_\BB)\omega_{\EE/\BB}\\
			& = & \sum\limits_{i}\int_{\UU_i\times \FF}\varphi_i\omega_\BB\omega_{\FF}\\
			& = &\sum\limits_i\int_\FF\left(\int_{\UU_i}\varphi_i\omega_\BB\right)\omega_{\FF}\\
			& = &\vol(\omega_\FF)\sum\limits_i\int_{\UU_i}\varphi_i\omega_\BB=\vol(\omega_\FF)\vol(\omega_\BB).
		\end{eqnarray*}
	\end{proof}

	\section{Sylow subgroups and $0$-groups}
	
	\subsection{Splitting subgroups} We recall the following definition from \cite{SSh}.
	
	\begin{definition}\label{definition splitting}
		Let $\HH\sub\GG$ be quasireductive supergroups.  Then we say that $\HH$ is splitting in $\GG$ if any of the following equivalent conditions hold:
		\begin{enumerate}
			\item $\C\langle 1\rangle$ splits off $\C[\GG/\HH]$ as a $\GG$-module.
			\item For every $\GG$-module $M$, the natural map of $\GG$-modules $M\to\Ind_{\HH}^{\GG}\operatorname{Res}^{\GG}_{\HH}(M)$ splits.
			\item The natural map $\Ext_{\GG}^i(M,N)\to\Ext_{\HH}^i(M,N)$ is injective for all $i$ and for all $\GG$-modules $M,N$.
			\item The natural map $\Ext_{\GG}^1(M,N)\to\Ext_{\HH}^1(M,N)$ is injective for all $\GG$-modules $M,N$.
			\item Every $\GG$-module $M$ is relatively projective over $\HH$.
		\end{enumerate}
	\end{definition}
	
	\begin{remark}
		Applying our definition of splitting to a finite group $G$ over a field of characteristic $p$, one sees that a subgroup $K$ of $G$ is splitting if and only if it contains a Sylow $p$-subgroup, or equivalently $\vol(G/K)=|G/K|\neq0$ (mod $p$).
		
		In particular, a Sylow $p$-subgroup is exactly a minimal splitting subgroup in characteristic $p$.
	\end{remark}
	
	The next lemma and the corollary following it are the main payoffs of our work on Berezin integration on CS manifolds.
	
	\begin{lemma}\label{lemma nonzero}
		If $\HH\sub\GG$ is quasireductive, then $\HH$ is splitting in $\GG$ if and only if $\Ber(\g/\h)$ is trivial and $\vol(\GG/\HH)\neq0$.
	\end{lemma}
	See Definition \ref{definition volume} for the meaning of the volume of $\GG/\HH$ being zero or not.
	\begin{proof}
		By Theorem \ref{thm existence integral}, if $\Ber(\g/\h)$ is trivial then we obtain a non-zero algebraic integral $\iota:\C[\GG/\HH]\to\C$.  By definition, $\vol(\GG/\HH)=\iota(1)$, so it is clear that if $\vol(\GG/\HH)\neq0$ we have that $\HH\sub\GG$ is splitting.
		
		Conversely, if $\HH\sub\GG$ is splitting we obtain by (1) of Definition \ref{definition splitting} a non-zero algebraic integral $\iota$ on $\GG/\HH$.  Once again by Theorem \ref{thm existence integral}, this implies that $\Ber(\g/\h)$ is trivial, thus giving rise to a non-zero algebraic integral from Berezin integration.  However by Lemma \ref{lemma integral unique}, all algebraic integrals are scalar multiples of $\iota$, meaning that \linebreak $\vol(\GG/\HH)=\iota(1)\neq0$. 
	\end{proof}
	
	\begin{cor}\label{cor  reflective transitivity}
		If $\KK\sub\HH\sub\GG$ are quasireductive supergroups, then $\KK$ is splitting in $\GG$ if and only if $\KK$ is splitting in $\HH$ and $\HH$ is splitting in $\GG$.  
	\end{cor} 
	
	\begin{proof}
		Using Cor.~2.7 of \cite{SSh}, we may assume that $\Ber(\g/\k)$, $\Ber(\g/\h)$, and $\Ber(\h/\k)$ are all trivial. 
		
		Let $\KK_c\sub\HH_c\sub\GG_c$ be compact CS forms of $\KK\sub\HH\sub\GG$.
		We set $\EE=\GG_c/\KK_c$, $\BB=\GG_c/\HH_c$, and $\FF=\HH_c/\KK_c$, so that we have a fibration $\EE\to \BB$ with fibre $\FF$.  Choose nonzero elements
		\[
		b_{\g/\h}\in\Ber(\g/\h), \ \ \ 
		b_{\h/\k}\in\Ber(\h/k), \ \ \ b_{\g/\k}:=b_{\g/\h}\cdot b_{\h/\k}\in\Ber(\g/\k). \]
		Here we have used that $\Ber(\g/\k)\cong\Ber(\g/\h)\otimes\Ber(\h/\k)$.
		
		From each $b_{\g/\k}$, $b_{\h/\k}$, and $b_{\g/\h}$ we obtain global Berezin forms $\omega_{\EE}$, $\omega_{\BB}$, and $\omega_{\FF}$ which trivialize their respective Berezin sheaves.  
		
		Now $\Ber(\Omega_{\EE/\BB})$ will be a $\GG$-equivariant bundle on $\GG/\KK$ with fibre naturally isomorphic to $\Ber(\h/\k)$, so we may trivialize it using $b_{\h/\k}$ to give a global trivialization $\omega_{\EE/\BB}$.  By equivariance, it is clear that $\omega_{\EE/\BB}|_{\FF}=\omega_\FF$ for any fibre, and thus we obtain $\omega_\EE=\pi^*(\omega_\BB)\omega_{\EE/\BB}$ as required to apply Proposition \ref{prop integral bundle}.  From this the corollary is clear.
	\end{proof}

	\begin{cor}\label{cor splitting product groups}
		Suppose that $\GG\cong\GG_1\times \GG_2$, and $\KK\sub \GG$ is splitting, where all supergroups are quasireductive.  Then $\KK\cap \GG_i\sub \GG_i$ is splitting.  
	\end{cor}
	
	\begin{proof}
		We show it for $i=1$.  First observe that $\KK\cap \GG_1$ is indeed quasireductive being the kernel of $\KK\to\GG_2$.
		
		Using Corollary \ref{cor  reflective transitivity}, $\KK$ is splitting in $\GG_1\times \pi_2(\KK)$. 
		We use that the quotient of a quasireductive group is quasireductive. 
		On the other hand, we have an isomorphism of quasireductive supergroups $$\GG_1\times \pi_2(\KK)/\KK\cong \GG_1/\KK\cap \GG_1,$$ 
		and since this isomorphism is $\GG_1$-equivariant, we learn that $\KK\cap \GG_1$ is splitting in $\GG_1$.
	\end{proof}

	\subsection{Sylow subgroups and $0$-groups} In analogy with finite groups, we make the following definition.
	
	\begin{definition}
		A Sylow subgroup of a quasireductive supergroup $\GG$ is a minimal splitting subgroup of $\GG$.
	\end{definition}
	
	Note that the existence of Sylow subgroups follows from the Zariski topology being Noetherian.
	
	\begin{remark}\label{remark reductive}
		For a quasireductive supergroup $\GG$, $\Rep\GG$ is semisimple if and only if the trivial subgroup is a Sylow subgroup.
	\end{remark}
	
	Recall that in the theory of finite-groups, a $p$-group $P$ is one for which $\vol(P/H)=|P/H|=0$ in $\F_{p}$ for all nontrivial subgroups $H$ in $P$.  Equivalently, $P$ contains no nontrivial splitting subgroups over characteristic $p$.
	
	\begin{definition}
		We say that a quasireductive supergroup $\OO$ is a 0-group if it satisfies any of the following equivalent conditions:
		\begin{enumerate}
			\item $\OO$ is its own Sylow subgroup.
			\item $\OO$ contains no nontrivial splitting subgroups.
			\item For every quasireductive subgroup $\KK\sub\OO$ such that $\Ber(\o/\k)$ is a trivial $\KK$-module, we have $\vol(\OO/\KK)=0$.
		\end{enumerate}
	\end{definition}

	\begin{remark}
		If $\OO$ is a $0$-group, then by Lem.~2.11 of \cite{SSh}, $\OO$ must be oddly generated (see Definition \ref{defn oddly generated}).  
	\end{remark}
	
	\begin{cor}
		Given quasireductive subgroups $\OO\sub\GG$, $\OO$ is a Sylow subgroup if and only if it is a splitting $0$-group.
	\end{cor}
	
	\begin{remark}
		One point of divergence from the theory of finite groups is that it is not true that every $0$-subgroup of a quasireductive supergroup $\GG$ will lie in a Sylow subgroup.  Indeed, if $\GG=\GG\LL(m|n)$, then the lie superalgebra has a $\Z$-grading $\g=\g_{-1}\oplus\g_0\oplus\g_1$, and if $mn>1$ then $\g_1$ will not Lie in any Sylow subgroup.
		
		The question of classifying all maximal 0-subgroups will be taken up in future work.
	\end{remark}
	
	\subsection{Splitting subalgebras}  In the following, we write $\Ext^i_{(\g,\g_{\ol{0}})}(-,-)$ for the $\Ext$ groups in the category $\Rep_{\g_{\ol{0}}}\g$ (see Definition \ref{defn qred algebra}).
	
	\begin{lemma}\label{lemma splitting subalgebras}
		Let $\g$ be quasireductive and $\k\sub \g$ a quasireductive subalgebra of $\g$.  Then the following are equivalent:
		\begin{enumerate}
			\item there exists global forms $\KK\sub\GG$ of $\k\sub\g$ such that $\KK$ is splitting in $\GG$;
			\item if $\KK\sub\GG$ are any global forms of $\k\sub\g$, then $\KK$ is splitting in $\GG$.
			\item $\Ext^i_{(\g,\g_{\ol{0}})}(M,N)\to \Ext^i_{(\k,\k_{\ol{0}})}(M,N)$ is injective for all $i$ and for all $M,N$ in $\Rep_{\g_{\ol{0}}}\g$;
			\item $\Ext^1_{(\g,\g_{\ol{0}})}(M,N)\to \Ext^1_{(\k,\k_{\ol{0}})}(M,N)$ is injective for all $M,N$ in $\Rep_{\g_{\ol{0}}}\g$.
		\end{enumerate}
	\end{lemma}
	\begin{proof}
		$(2)\Rightarrow(1)$ and $(4)\Rightarrow(3)$ are clear, as is $(3)\Rightarrow(2)$ by Lemma \ref{lemma full subcat}. 
		For $(1)\Rightarrow(4)$, it suffices to show that $\Ext^\bullet_{(\g,\g_{\ol{0}})}(M,N)\to \Ext^\bullet_{(\k,\k_{\ol{0}})}(M,N)$ is injective when $N$ is a finite-dimensional $\g$-module, and thus we may assume that $N=\C$ by tensoring $M$ by $N^*$.
		
		We may write $M=M'\oplus M''$, where $M'$ lies in the principal block of $(\g,\g_{\ol{0}})$, and $M''$ has no simple constituents lying in the principal block. Thus \linebreak $\Ext^i_{(\g,\g_{\ol{0}})}(M,\C)=\Ext^i_{(\g,\g_{\ol{0}})}(M',\C)$, and so we may assume $M$ lies in the principal block.  But the principal block of $\Rep_{\g_{\ol{0}}}\g$ lies in $\Rep\GG$ by Lemma \ref{lemma full subcat}, so we are done.
	\end{proof}
	
	\begin{definition}
		If $\k\sub\g$ are quasireductive, then we say that $\k$ is splitting in $\g$ 
		if any of the equivalent conditions of Lemma \ref{lemma splitting subalgebras} holds.
	\end{definition}
	
	The following is a consequence of Corollary \ref{cor  reflective transitivity} and Corollary \ref{cor splitting product groups}.
	\begin{cor}\label{cor product}
		\begin{enumerate}
			\item If $\k\sub\h\sub\g$ are quasireductive, then $\k$ is splitting in $\g$ if and only if $\k$ is splitting in $\h$ and $\h$ is splitting in $\g$.
			\item If $\h\sub\g_1\times\g_2$ is splitting, then $\h\cap\g_i$ is splitting in $\g_i$. 
		\end{enumerate} 
	\end{cor}
	
	\subsection{Sylow subalgebras and 0-superalgebras}  
	\begin{definition}
		Let $\g$ be quasireductive.  We say that a quasireductive subalgebra $\o\sub\g$ is a Sylow subalgebra if $\o$ is splitting in $\g$, and is minimal with this property.
	\end{definition}
	\begin{definition}
		We say that a Lie superalgebra $\o$ is a $0$-superalgebra if it is a Sylow subalgebra of itself.
	\end{definition}
	
	\begin{remark}
		If $\o$ is a $0$-superalgebra, then it must be oddly generated (see Definition \ref{defn oddly generated}). 
	\end{remark}
	\begin{lemma}\label{lemma 0-algebra 0-group}
		If $\o$ is a 0-superalgebra and $\OO$ is a connected global form of $\o$, then $\OO$ is a 0-group.
	\end{lemma}
	\begin{proof}
		Let $\KK\sub\OO$ be a Sylow subgroup.  By Lemma \ref{lemma splitting subalgebras}, this implies that $\k=\operatorname{Lie}\KK$ is splitting in $\operatorname{Lie}\OO$.  On the other hand, since $\k_{\ol{1}}+[\k_{\ol{1}},\k_{\ol{1}}]$ is splitting in $\k$ and is contained in $\o$, it must also be splitting inside $\o$.  Since $\o$ is a $0$-superalgebra, we obtain that 
		$\k=\o$, which in turn implies $\KK=\OO$.
	\end{proof}
	
	The following Lemma follows immediately from \ref{lemma splitting subalgebras} and \ref{lemma 0-algebra 0-group}.
	\begin{lemma}\label{lemma sylow subalg gives sylow subgp}
		Suppose that $\o\sub\g$ are quasireductive Lie superalgebras such that $\o$ is Sylow in $\g$.  If $\OO\sub\GG$ are global forms of $\o$ and $\g$, then $\OO$ is Sylow in $\GG$.
	\end{lemma}

	\begin{cor}
		If $\g$ is quasireductive, then a quasireductive subalgebra $\o\sub\g$ is Sylow if and only if it is a splitting $0$-subalgebra.
	\end{cor}
	
	\subsection{Bijective correspondence of 0-subalgebras and 0-subgroups}
	
	\begin{lemma}\label{lemma bijection 0 subalgs and subgps}
		Let $\g$ be a Lie superalgebra with global form $\GG$. 
		\begin{enumerate}
			\item The bijection $\mathfrak{G}$ of Lemma \ref{lemma bijection oddly generated} defines a $\GG_0$-equivariant bijection between the 0-subalgebras of $\g$ and the 0-subgroups of $\GG$.  
			\item The bijection $\mathfrak{G}$ restricts to a $\GG_0$-equivariant bijection between the Sylow subalgebras of $\g$ and the Sylow subgroups of $\GG$.
			\item Further, $\mathfrak{G}$ restricts to $\GG_0$-equivariant bijection between the maximal $0$-subalgebras of $\g$ and the maximal $0$-subgroups of $\GG$.
		\end{enumerate}
	\end{lemma}
	\begin{proof}
		This is a straightforward consequence of Lemmas \ref{lemma 0-algebra 0-group}, \ref{lemma sylow subalg gives sylow subgp} and \ref{lemma full subcat}.
	\end{proof}
	\subsection{Neat and homological elements} Let $\GG$ be quasireductive.  We recall from \cite{ES} the following definition.

	\begin{definition}
		An element $x\in\g_{\ol{1}}$ is called neat if either $x=0$ or if there exists a subalgebra $\o\s\p(1|2)\cong\k\sub\g$ with $x\in\k_{\ol{1}}$.
	\end{definition}
	
	Note that the definition of neat elements in \cite{ES} is different than the one given here, although is shown to be equivalent in \textit{loc.~cit.}
	
	\begin{definition}
		The homological elements of $\g_{\ol{1}}$ are defined as
		\[
		\g_{\ol{1}}^{hom}=\{x\in\g_{\ol{1}}:[x,x]\text{ is semisimple in }\g_{\ol{0}}\}.
		\]
	\end{definition} 
	
	By Thm.~2.15 of \cite{SSh}, if $\KK\sub\GG$ is splitting then we must have 
	\begin{equation}\label{equation splitting homological}
		\GG_0\cdot\k_{\ol{1}}^{hom}=\g_{\ol{1}}^{hom},
	\end{equation}
	where $\cdot$ stands for the adjoint action. 
	In particular, if $\Rep\GG$ is semisimple, then necessarily $\g_{\ol{1}}^{hom}=\{0\}$.  We now show the converse, which can be viewed as a version of Cauchy's lemma for supergroups.

	\begin{prop}\label{prop semisimple criterion} The following conditions on a quasireductive Lie superalgebra
		$\g$ are equivalent:
		\begin{enumerate}
			\item $\g_{\ol{1}}^{hom}=\{0\}$;
			\item Every odd element of $\g$ is neat;
			\item $\Rep_{\g_{\ol{0}}}\g$ is semisimple.    
		\end{enumerate}
	\end{prop}
	\begin{proof} Let us show first that (1) implies (2). Let $x\in\g_{\ol{1}}$, $x\neq 0$ and $[x,x]=y_s+y_n$
		be the Jordan decomposition of $y=[x,x]$. By our assumption $y_n\neq 0$. Therefore by the Jacobson-Morozov theorem
		there exists an $\s\l(2)$-triple $\{e,h,f\}\subset\g^{y_s}$ such that $e=y_n$. Since $[e,x]=0$ we can write
		$x=\sum_{k=0}^lx_k$ such that $[h,x_k]=kx_k$. We first note that $[x_0,x_0]=y_s$ and hence by (1)
		$x_0=y_s=0$. The relations $[h,e]=2e$ and $[x,x]=e$ imply $[x_k,x_k]=0$ for $k>1$ and hence by (1) we get
		$x=x_1$. Now one can easily see that $\{e,h,f,x,[f,x]\}$ form a basis of the $\o\s\p(1|2)$ subalgebra in $\g$.
		
		Let us now show that (2) implies (3). Since $\o\s\p(1|2)$ is simple we have that if (2) holds for $\g$
		and $\r\subset\g$ is an ideal, then (2) holds for $\r$ and $\g/\r$. On the other hand, if (3) holds for
		$\r$ and $\g/\r$ then $\r$ is splitting for $\g$ and therefore (3) holds for $\g$. Now we can proceed by induction on dimension of $\g$ and we have to check the statement only for simple Lie superalgebra $\g$.  Now the statement follows from the classification of simple Lie superalgebras (\cite{K}), which shows that only $\g=\o\s\p(1|2n)$ has that $\g_{\ol{1}}^{neat}=\g_{\ol{1}}$.
		
		Finally, (3) implies (1) by the argument before the proposition.
	\end{proof}
	
	\begin{cor}
		If $\GG$ is a quasireductive supergroup such that $\g_{\ol{1}}^{hom}=\{0\}$, then $\Rep\GG$ is semisimple.
	\end{cor}

	\subsection{A priori results on Sylow subgroups and 0-groups}
	\begin{lemma}\label{lemma 0 group properties} Let $\OO$ be a $0$-group with Lie superalgebra $\o$.
		\begin{enumerate}
			\item Any quotient of $\OO$ is again a $0$-group.
			\item $\OO$ does not admit a simple quotient.
			\item $\OO$ does not admit a nontrivial, purely even quotient.
			\item $\o$ does not have simple ideal.
			\item The product of two $0$-groups is again a $0$-group.
		\end{enumerate}
	\end{lemma}
	
	\begin{proof}
		For (1), if $\OO$ is a $0$-group and $\pi:\OO\to\OO'$ is a quotient of $\OO$ with nontrivial subgroup $\KK\sub\OO'$, then 
		\[
		\vol(\OO'/\KK)\cong\vol(\OO/\pi^{-1}(\KK))=0,
		\]
		which implies that $\OO'$ is a $0$-group.
		
		For (2), we use (1) and the results of \cite{SSh} and \cite{SV}, whose main theorems show that any simple quasireductive supergroup admits a nontrivial splitting subgroup.  
		
		Once again (3) follows from (1) because a nontrivial reductive group is never a $0$-group.
		
		For (4) assume $\s$ is a simple ideal in $\o$. Let $\mathfrak{i}$ be a maximal ideal in $\o$ which intersects 
		$\s$ by zero and set $\bar\o:=\o/\mathfrak{i}$. Note that $\bar \o$ is the Lie superalgebra of a 0-group. On the other hand,
		$\o$ is a subalgebra of $\operatorname{der}\s$. By \cite{S} (see also Lemma \ref{lemma derivations simples}, $\operatorname{der}\s/\s$ is a purely even algebra
		unless $\s=\p\s\q(n)$ for $n\geq 3$ and $\operatorname{der}\s=\p\q(n)$. Therefore $\bar\o=\s$ unless 
		$\bar \o=\p\q(n)$.
		The former case is impossible by (2) and the latter case is also impossible, see \cite{SSh}.
		
		Finally, (5) follows from Corollary \ref{cor splitting product groups}.
	\end{proof}
	
	\section{Classification and structure of $0$-groups}
	
	In this section, we classify $0$-groups and give several equivalent characterizations. 
	
	\subsection{Takiff superalgebras and supergroups}\label{section takiff}
	
	\begin{definition}
		We call a Lie superalgebra $\s$ a  Takiff superalgebra,  if there exist simple (even) Lie algebras $\s_1,\dots,\s_k$ such that
		\[
		\s\cong\bigoplus\limits_{i}\s_i\otimes\C[\xi_i],
		\]
		where each $\mathbb C[\xi_i]$ is a supercommutative algebra with one odd generator $\xi_i$.
	\end{definition}
	
	Observe that a Takiff superalgebra is never a 0-superalgebra, as it is not oddly generated.
	
	\begin{definition} A Lie superalgebra $\g$ is called a Takiff 0-superalgebra if $\g$ is semisimple
		(has trivial radical), $\g$ is oddly generated (i.e.~$[\g_{\ol{1}},\g_{\ol{1}}]=\g_{\ol{0}}$),
		and $[\g,\g]$ is a Takiff superalgebra. It is equivalent to the condition that
		\[
		\bigoplus\limits_{i}\s_i\otimes\C[\xi_i]\sub\g\sub\bigoplus\limits_{i}\s_i\otimes\C[\xi_i]\rtimes\C\langle\d_{\xi_i}\rangle,
		\]
		and that $[\g_{\ol{1}},\g_{\ol{1}}]=\g_{\ol{0}}$.  Clearly, we may write $\g=\bigoplus\limits_{i}\s_i\otimes\C[\xi_i]\rtimes\mathfrak{d}$, where $\mathfrak{d}\sub\C\langle\d_{\xi_1},\dots,\d_{\xi_k}\rangle$.  Then the condition $[\g_{\ol{1}},\g_{\ol{1}}]=\g_{\ol{0}}$ is equivalent to asking that the natural projections $\mathfrak{d}\to\C\langle\d_{\xi_i}\rangle$ are surjective for all $i$.
	\end{definition}
	
	\begin{definition}
		A supergroup $\mathcal{S}$ is called a Takiff supergroup if $\s=\operatorname{Lie}\mathcal{S}$ is a Takiff superalgebra. 
		Similarly, a supergroup $\GG$ is called a Takiff 0-supergroup if $\g=\operatorname{Lie}\GG$ is a Takiff 0-superalgebra.
	\end{definition}
	
	\begin{example}
		Examples of Takiff 0-superalgebras include:
		\[
		(\s\l_2\otimes\C[\xi_1]\oplus \s\l_3\otimes\C[\xi_2])\rtimes\C\langle\d_{\xi_1}-\d_{\xi_2}\rangle.
		\]
		Another example is 
		\[
		\p\q(2)\cong\s\mathfrak{pe}(2)\cong(\s\l(2)\otimes\C[\xi])\rtimes\C\langle\d_{\xi}\rangle,
		\]
		where $\s\mathfrak{pe}(2)$ is the commutator subalgebra of the periplectic Lie superalgebra $\mathfrak{pe}(2)$, and $\p\q(2)$ is the quotient of the queer Lie superalgebra $\q(2)$ by its centre.
		
		A non-example of a Takiff 0-superalgebra is
		\[
		(\s\l_2\otimes\C[\xi_1]\oplus \s\l_3\otimes\C[\xi_2])\rtimes\C\langle\d_{\xi_1}\rangle.        \]
	\end{example}
	
	\begin{remark}\label{remark derivations Takiff}
		For a Takiff superalgebra $\s=\bigoplus\limits_{i=1}^k\s_i\otimes\C[\xi]$, we have:
		\[
		\operatorname{Der}(\s)/\s=\langle\d_{\xi_1},\dots,\d_{\xi_k},\xi_1\d_{\xi_1},\dots,\xi_{k}\d_{\xi_k}\rangle.
		\]
		See, for instance,  Lem.~6.6 of \cite{S}.  
	\end{remark}
	
	\subsection{Classification theorem} 
	In the following, a central extension may be of arbitrary size.
	
	\begin{thm}\label{thm classification 0 groups}
		A quasireductive supergroup $\OO$ is a 0-group if and only if $\OO$ is oddly generated and isomorphic to a central extension of $\TT\times \mathcal{V}$, where $\mathcal{V}$ is an odd abelian supergroup and $\TT$ is a Takiff 0-supergroup.
	\end{thm}
	We note that by Lemma \ref{lemma bijection 0 subalgs and subgps}, Theorem \ref{thm classification 0 groups} is equivalent to the statement that a quasireductive Lie superalgebra $\o$ is a $0$-superalgebra if and only if it is oddly generated and isomorphic to a central extension of $\mathfrak{t}\times\mathfrak{v}$, where $\mathfrak{t}$ is a Takiff 0-superalgebra, and $\mathfrak{v}$ is an odd Lie abelian superalgebra.
	
	For the backwards direction of Theorem \ref{thm classification 0 groups}, we will need the following lemma:
	
	\begin{lemma}\label{lemma semidirect product}
		Let $\GG=\GG'\ltimes\mathcal{D}$ where $\GG,\GG'$ are quasireductive, and $\mathcal{D}$ is an odd abelian supergroup that commutes with $\GG_0$.  Then $\mathcal{D}$ is contained in every splitting subgroup of $\GG$.  Further, $\KK=\KK'\ltimes \mathcal{D}$ is splitting in $\GG$ if and only if $\KK'$ is splitting in $\GG'$.
	\end{lemma}
	
	\begin{proof}
		Since $\operatorname{Lie}\DD$ consists of homological, $\GG_0$-fixed vectors, the first claim follows from \ref{equation splitting homological}.
		
		For the second statement, observe that the natural map $\GG'/\KK'\to \GG/\KK$ is an isomorphism of $\GG'$-varieties.  We clealy have a natural isomorphism $\Ber(\g/\k)\cong\Ber(\g'/\k')$, and since $\KK_0'=\KK_0$, $\Ber(\g/\k)$ is trivial as a $\KK$-module if and only $\Ber(\g'/\k')$ is trivial as a $\KK'$-module.
		
		Thus let us assume these Berezin modules are trivial, meaning both $\GG/\KK$ and $\GG'/\KK'$ admit global, invariant, Berezin forms $\omega_{\GG/\KK}$ and $\omega_{\GG'/\KK'}$.  Since $\omega_{\GG/\KK}$ pulls back to $\omega_{\GG'/\KK'}$ under the isomorphism $\GG'/\KK'\cong\GG/\KK$, it is clear that $\vol(\GG/\KK)\neq0$ if and only if $\vol(\GG'/\KK')\neq0$, so we may conclude by Lemma \ref{lemma nonzero}.
	\end{proof}
	
	\begin{proof}[Proof of Theorem \ref{thm classification 0 groups}]
		Let $\OO$ be oddly generated and a central extension of $\TT\times\VV$, where $\TT$ is a Takiff 0-supergroup and $\VV$ is odd abelian.  To show that $\OO$ is a 0-group, we may quotient by the centre, and thus assume that $\OO\cong\TT\times\VV$.  However this implies $\OO\cong\mathcal{S}\ltimes\mathcal{D}$, where $\mathcal{S}$ is a Takiff supergroup and $\mathcal{D}$ is an odd abelian subgroup that commutes with $\OO_0$.  Thus we may apply Lemma \ref{lemma semidirect product}, which implies that any splitting subgroup $\KK$ of $\OO$ must be of the form $\KK'\ltimes\mathcal{D}$, where $\KK'\sub\mathcal{S}$ is splitting. However, $\s_{\ol{1}}$ is an odd abelian ideal of $\s$, implying that $\s_{\ol{1}}\sub\k'$.  Thus $\mathfrak{o}_{\ol{1}}=\k_{\ol{1}}$, implying $\KK=\OO$.
		
		For the converse direction, we apply Lemma \ref{lemma 0 group properties} and the following Lemma \ref{lemma no simple ideals}.
	\end{proof}
	
	\begin{lemma}\label{lemma no simple ideals}
		Let $\OO$ be an oddly generated quasireductive supergroup with Lie superalgebra $\o$ such that $\o/\mathfrak{z}(\o)_{\ol{0}}$ contains no simple ideal.  Then $\OO$ is isomorphic to a central extension of $\TT\times\VV$, where $\TT$ is a Takiff 0-supergroup and $\VV$ is odd abelian.
	\end{lemma}
	
	\begin{proof}
		Let $\OO$ be a such a supergroup, and assume without loss of generality that $\OO$ has trivial centre. Thus we want to show that $\OO=\TT\times\mathcal{V}$, where $\TT$ is a Takiff 0-supergroup and $\VV$ is odd abelian.  
		
		By Lem.~5.6 of \cite{S}, the minimal ideals of $\o=\operatorname{Lie}\OO$ are either simple, odd abelian, or a Takiff superalgebra.  Let $\c(\o)$ be the product of minimal ideals of $\o$. Since by assumption $\o$ does not have simple ideals,  by Thm.~6.9 of \cite{S}, $\mathfrak{c}(\o)$ is isomorphic to $\mathfrak{s}\times\mathfrak{v}$ for a Takiff superalgebra $\s$ and an odd abelian ideal $\mathfrak{v}$.  Further, $\o/\c(\o)=\r\ltimes\mathfrak{l}$ for a reductive Lie algebra $\r$ and an odd abelian ideal $\mathfrak{l}$. However $[\o_{\ol{1}},\o_{\ol{1}}]\sub\c(\o)$, so we must have $\r=0$.  
		
		Thus we have $\o=(\s\times\mathfrak{v})\rtimes \mathfrak{l}$. Hence: 
		\[
		\l\sub\operatorname{Der}(\s\times\mathfrak{v})_{\ol{1}}.
		\]
		Clearly $[\l,\mathfrak{v}]=0$.  By Remark \ref{remark derivations Takiff}, we are done.
	\end{proof}
	
	\subsection{Characterization via neat elements and the nil-cone}\label{section char 0 gps}  Let $\GG$ be quasireductive with Lie superalgebra $\g$.  We now recall three definitions, the first from \cite{ES} and the third from \cite{GHSS}.

	\begin{definition}
		We define the nil-cone of $N(\g_{\ol{1}})\sub \g_{\ol{1}}$ to be the set of $x\in\g_{\ol{1}}$ such that 0 lies in the closure of $\GG_0\cdot x$ in $\g_{\ol{1}}$.
	\end{definition}
	Observe that if $\KK\sub\GG$ are quasireductive supergroups, then $N(\k_{\ol{1}})\sub N(\g_{\ol{1}})$.  
	
	\begin{definition}
		For a Lie superalgebra $\g$, set $X_{\g}:=\{x\in\g_{\ol{1}}:[x,x]=0\}$.  This is the cone of odd self-commuting elements of $\g$ (see \cite{GHSS}).
	\end{definition} 
	Observe that $\g_{\ol{1}}^{hom}\cap N(\g_{\ol{1}})\sub X_{\g}$.

	\begin{prop}\label{prop 0 characterization}
		Let $\OO$ be an oddly generated quasireductive supergroup.  Then the following are equivalent:
		\begin{enumerate}
			\item $\OO$ is a $0$-group;
			\item $N(\o_{\ol{1}})\sub X_{\o}$.
			\item $N(\o_{\ol{1}})\sub\o_{\ol{1}}^{hom}$;
			\item $\o_{\ol{1}}^{neat}=\{0\}$.
		\end{enumerate}
	\end{prop}
	\begin{proof}
		(1)$\Rightarrow$(2) Let $\OO$ be a $0$-group.  Then after quotienting by the center, $\mathfrak{o}=\mathfrak{s}\rtimes\mathfrak{d}$, where $\mathfrak{s}$ is a Takiff superalgebra and $\mathfrak{d}$ is an odd abelian subalgebra fixed by $\OO_0$.  Because $\mathfrak{d}$ is fixed by $\OO_0$, we see that $N(\o_{\ol{1}})\sub\s_{\ol{1}}$.  However $\s_{\ol{1}}\sub X_{\o}$, as desired. (2)$\Rightarrow(3)$ is obvious

		For $(3)\Rightarrow (4)$, if $\g_{\ol{1}}^{neat}\neq\{0\}$ then there exists an embedding $\o\s\p(1|2)\sub\g$.  But $N(\o\s\p(1|2)_{\ol{1}})=\o\s\p(1|2)_{\ol{1}}$, while $\o\s\p(1|2)^{hom}_{\ol{1}}=\{0\}$, giving a contradiction.

		(4)$\Rightarrow$(1) Suppose that $\o_{\ol{1}}^{neat}=\{0\}$.  By a check, one can show that any simple quasireductive Lie superalgebra admits non-zero neat elements.  In particular, $\o$ does not contain any simple ideal. We conclude by Lemma \ref{lemma no simple ideals}. 
	\end{proof}
	
	\begin{cor}\label{cor sub of 0 is 0}
		Let $\o$ be a $0$-superalgebra, and let $\o'\sub\o$ be an oddly generated, quasireductive subalgebra.  Then $\o'$ is again a $0$-superalgebra.
	\end{cor}
	\begin{proof}
		This follows easily from $(\o_{\ol{1}}')^{neat}\sub\o_{\ol{1}}^{neat}=\{0\}$ and part (4) of Proposition \ref{prop 0 characterization}.
	\end{proof}
	\subsection{Root decomposition for $0$-groups}
	
	\begin{lemma}\label{lem rootzero} Let $\o$ be a $0$-superalgebra. Then there exists $x\in\o_{\bar 1}^{hom}$ such that:
		\begin{enumerate}
			\item $\h:=\c_{\o}(x^2)$ is a Cartan subalgebra of $\o$.
			\item $[\h_{\bar 1},\h_{\bar 1}]=\h_{\bar 0}$.
			\item The root system $\Delta\sub\h_{\ol{0}}^*$ coincides with the root system of $\o_{\bar{0}}$ and hence is a classical reduced root system.
			\item For any root $\alpha\in\Delta$, the subalgebra $\o^{\alpha}$ generated by root spaces $\o_{\pm\alpha}$ is isomorphic to
			$\p\s\q(2)$ or $\s\q(2)$.  
			\item For any root $\alpha\in\Delta$, $\o_{\alpha}$ is an irreducible $\h$-module.
		\end{enumerate}
	\end{lemma}
	\begin{proof} This follows directly from the description of $0$-groups in Theorem \ref{thm classification 0 groups}, and the fact that the same is true for Takiff 0-superalgebras.
	\end{proof}
	
	\begin{remark} We remark that in the language of \cite{ShSi}, Lemma \ref{lem rootzero} shows that $0$-superalgebras are all queer Kac-Moody.  In fact, it was shown in \textit{loc. cit.} that all finite type queer Kac-Moody superalgebras are products of Takiff 0-superalgebras and the queer Lie superalgebra $\q(n)$ (up to central extensions).
	\end{remark}

	\begin{lemma}\label{lem simplezero} Let $\o$ be a $0$-superalgebra and let $L\in\Rep_{\o_{\ol{0}}}\o$ be a simple module. Let $\PP(L)$ denote the set of weights of $L$. 
		\begin{enumerate}
			\item If the center of $\o$ annihilates $L$, then $L\not\cong \Pi L$, and $\operatorname{Res}_{\o_{\ol{0}}}L$ is a direct sum of several copies of one simple
			$\o_{\bar 0}$-module.
			\item If $L$ is nontrivial then the superdimension of any weight space is zero.
			\item If $0\notin \PP(L)$ then $\Ext^1_{(\o,\o_{\ol{0}})}(\mathbb C,L)=0$.
			\item If $L,L'$ are simple modules such that $\PP(L)\cap\PP'(L)=\emptyset$ then
			$\Ext_{(\o,\o_{\ol{0}})}^1(L,L')=0$.
			\item If  the multiplicity of any weight space of $L$ equals $(1|1)$, then $L$ is not isomorphic to
			$L^*$.
		\end{enumerate}
	\end{lemma}
	\begin{proof} To prove (1) we may assume without loss of generality that the center of $\o$ is trivial.  Then if $\h\sub\o$ denotes a Cartan subalgebra, we have a decomposition $\h=\h_1\oplus\h_2$, where $[\h_1,\h_1]=[\h_2,\h_2]=0$.  Thus any simple $\h$-module will not be isomorphic to its parity shift, implying the same for any simple $\o$-module by highest weight theory.
		
		Now write $\o_{\bar 1}=\m\oplus \o_{\bar 1}^{\o_{\bar 0}}$. Then $\m$ is a purely odd abelian Lie superalgebra, and $L^{\m}\neq 0$ is $\o_{\bar 0}$-invariant. Let $M_0$ be a simple $\o_{\bar 0}$-submodule of
		$L^{\m}$. We have that $L$ is a quotient of $U(\o)\otimes_{U(\o_{\bar 0}+\m)}M_0$. Since all simple
		$\o_{\bar 0}$-components of the latter module are isomorphic, this completes the proof of (1).
		
		Let us prove (2). For any nonzero weight,  the statement follows from Lemma \ref{lem rootzero} (2).
		Therefore (2) holds if $L$ is not trivial over the center of $\o$.  If the center acts trivially on $L$ consider the highest weight
		space of $L$. If the highest weight is not zero, then the highest weight space has superdimension $0$ and the statement follows from (1).
		
		For (3), by Lemma \ref{lem rootzero} the root lattice of $\o$ is equal to the root lattice of $\o_{\ol{0}}$, implying a corresponding decomposition of $\Rep_{\ol{0}}\o$ into modules with weights lying in a given coset of the weight lattice modulo the root lattice.  This implies the result.
		
		
		(4) follows from (3) by using $\Ext_{(\o,\o_{\ol{0}})}^1(L,L')=\Ext^1_{(\o,\o_{\ol{0}})}(\mathbb C,L^*\otimes L')$. 
		
		(5) If $L\simeq L^*$ then the center acts trivially on $L$. The condition on weight multiplicity and
		(1) ensures that $L_{\bar 0}\simeq L_{\bar 1}$ is a simple $\o_{\bar 0}$-module. On the other hand, since $L\not\cong\Pi L$, $L$ must admit an invariant even symmetric or skew symmetric form, which would mean that $L_{\bar 0}$
		admits both a symmetric and skew-symmetric form, which is impossible.
	\end{proof}

	\section{Uniqueness of Sylow subgroups up to conjugacy}
	
	\begin{thm}\label{theorem uniqueness conjugacy}
		Let $\GG$ be a quasireductive supergroup.  Then any two Sylow subgroups $\OO,\OO'$ of $\GG$ are conjugate under $\GG_0$.
	\end{thm}
	
	The proof of Theorem \ref{theorem uniqueness conjugacy} will be given in steps, and will occupy the rest of the section.  Write $\g=\operatorname{Lie}\GG$.  
	
	Using Lemma \ref{lemma bijection 0 subalgs and subgps}, we easily obtain that Theorem \ref{theorem uniqueness conjugacy} holds if and only all Sylow subalgebras of $\g$ are conjugate.  Thus we sometimes work with $\g$ and sometime with $\GG$, as is convenient.
	
	We start with a useful lemma.
	
	\begin{lemma}\label{lemma cent extn even deriv}
		Let $\GG$ be a quasireductive supergroup, and suppose that $\tilde{\GG}=\GG_c\rtimes\DD$, where $\GG_c$ is an even central extension of $\GG$ and $\DD=\DD_0$ is even reductive.  Then the Sylow subgroups of $\GG$ are in natural bijection with the Sylow subgroups of $\tilde{\GG}$, and this bijection respects conjugation by $\GG_0$.
	\end{lemma}
	
	\begin{proof}
		Write $\pi:\GG_c\to\GG$ for the natural quotient.  Using Lem.~2.9 of \cite{SSh}, it is easy to check that a subgroup $\OO\sub\GG_c$ is Sylow if and only if $\pi(\OO)$ is Sylow in $\GG$, and this correspondence clearly respects conjugation.  Therefore we may assume $\GG_c=\GG$, so that $\tilde{\GG}=\GG\rtimes\DD$.  
		
		Because $\DD$ is reductive, it is clear that every Sylow subgroup of $\tilde{\GG}$ is contained, and thus also a Sylow subgroup, of $\GG$.  Conversely, since $\GG$ is splitting in $\tilde{\GG}$, a Sylow subgroup in $\GG$ is also a Sylow subgroup in $\tilde{\GG}$, giving our conjugation-preserving bijection.
	\end{proof}
	
	\subsection{Uniqueness up to conjugacy when $\g$ is simple} The simple Lie superalgebras were classified in \cite{K}.  All quasireductive simple Lie superalgebras appear in the following table.  The work in \cite{SSh} and \cite{SV} implies that we have the following table of Sylow subalgebras of the (almost) simple superalgebras:

	\renewcommand{\arraystretch}{1.5}	
	\begin{center}
		\begin{tabular}{|c|c|}
			\hline 
			$\g$ & Sylow subalgebra $\o$ \\
			\hline   
			$\g$ Kac Moody of defect $d$ & $\s\l(1|1)^d$\\
			\hline
			$\p\s\l(d|d)$ & $\p(\s\l(1|1)^d)$\\ 
			\hline
			$\mathfrak{pe}(2n+1)$ ,$\s\p\mathfrak{e}(2n+1)$ & $\s\p\mathfrak{e}(2)^n\times\s\p\mathfrak{e}(1)$ \\
			\hline
			$\mathfrak{pe}(2n),\s\p\mathfrak{e}(2n)$ & $\s\p\mathfrak{e}(2)^n$  \\
			\hline
			$\p\s\q(2n+1)$ & $\p\s(\q(2)^n\times\q(1))$\\
			\hline 
			$\p\s\q(2n)$ & $\p\s(\q(2)^n)$\\
			\hline 
		\end{tabular}
	\end{center}
	
	Some explanations: 
	\begin{enumerate}
		\item we write $\mathfrak{pe}(n)$ for the  periplectic Lie superalgebra, and $\mathfrak{spe}(n)$ for its derived subalgebra;
		\item $\p(\s\l(1|1)^d)$ denotes the quotient of $\s\l(1|1)^d$ by the one-dimensional central ideal spanned by $(h,\dots,h)$, where $h\in\s\l(1|1)$ is a non-zero central element.
		\item $\s\q(n)$ denotes the commutator subalgebra of $\q(n)$, and $\p\s\q(n)$ is the quotient of $\s\q(n)$ by its one-dimensional centre.
		\item $\s(\q(2)^n(\times\q(1)))$ denotes the odd codimension-one subalgebra of $\q(2)^n(\times\q(1))$ given by the kernel of the superalgebra homomorphism $\phi:\q(2)^n\times\q(1)\to\C^{0|1}$, where $\phi$ the sum of the odd trace maps from each factor. 
		\item $\p\s(\q(2)^n(\times\q(1)))$ is the quotient of $\s(\q(2)^n(\times\q(1)))$ by the subalgebra spanned by $(c,\dots,c)$, where $c$ is a nonzero central element in each factor $\q(2)$ or $\q(1)$, respectively.
	\end{enumerate}

	\begin{remark}
		The roots of the Sylow subalgebras of simple Lie superalgebras in the above table always form an iso-set, as defined in \cite{G}.
	\end{remark}
	
	The following lemma is proven in, for instance, Section 6 of \cite{S}.  
	\begin{lemma}\label{lemma derivations simples}
		Let $\g$ be a quasireductive, simple Lie superalgebra, and let $\operatorname{Out}(\g)=\operatorname{Der}(\g)/\g$.  
		
		\begin{enumerate}
			\item We have a splitting $\operatorname{Out}(\g)\sub\operatorname{Der}(\g)^{\g_{\ol{0}}}$, and $\operatorname{Out}(\g)=\operatorname{Der}(\g)^{\g_{\ol{0}}}$ unless $\g=\s\l(m|n)$ for $m\neq n$, $mn>1$, or $\g=\o\s\p(2|2n)$ for $n>0$.
			\item Further, we have:
			\begin{enumerate}
				\item $\operatorname{Der}(\g)^{\g_{\ol{0}}}=0$ when $\g=\o\s\p(m|2n)$ for $m\neq 2$, and for $\g=\mathfrak{d}(2|1;a)$, $\a\g(1|2)$, $\a\b(1|3)$;
				\item $\operatorname{Der}(\g)^{\g_{\ol{0}}}=\C\langle h\rangle$ for $h$ an even grading operator, when $\g=\o\s\p(2|2n)$ for $n>0$, $\g=\s\l(m|n)$ for $m\neq n$, $mn>1$, $\g=\p\s\l(n|n)$ for $n\geq 3$, and $\g=\mathfrak{spe}(n)$ for $n\geq 3$;
				\item $\operatorname{Der}(\p\s\l(2|2))^{\p\s\l(2|2)_{\ol{0}}}\cong \s\l(2)$;
				\item $\operatorname{Der}(\p\s\q(n))^{\p\s\q(n)_{\ol{0}}}=\C\langle H\rangle$, where  $H$ is odd, and $\ad(H):\p\s\q(n)_{\ol{1}}\to\p\s\q(n)_{\ol{0}}$ is an isomorphism.
			\end{enumerate}
		\end{enumerate}  
	\end{lemma}
	
	\begin{prop}\label{prop simple conj unique}
		Sylow subalgebras of simple quasireductive Lie superalgebras are unique up to conjugacy and stable under $\operatorname{Der}(\g)^{\g_{\ol{0}}}$.  In particular, they are stable under $\operatorname{Out}(\g)$.
	\end{prop}
	
	We now prove Proposition \ref{prop simple conj unique} by showing that in every case listed in our table above, the Sylow subalgebra presented is unique up to conjugacy.  Note that one can check directly that each such subalgebra is stable under $\operatorname{Der}(\g)^{\g_{\ol{0}}}$, implying the second half of Proposition \ref{prop simple conj unique}.
	
	It will sometimes be convenient to work not with the simple Lie superalgebra $\g$, but rather a larger Lie superalgebra $\tilde{\g}$ which is obtained from $\g$ via an even central extension and adding even derivations, as in Lemma \ref{lemma cent extn even deriv}.  From this we use that Sylow subalgebras are stable under the even outer derivations, so that the bijection from Lemma \ref{lemma cent extn even deriv} will descend to a bijection on conjugacy classes.

	\subsubsection{Queer superalgebra $\p\s\q(n)$}   We have that $\p\q(n)=\p\s\q(n)\rtimes\C\langle H\rangle$, where $H\in\p\q(n)_{\ol{1}}$ commutes with the even part and has that $[H,H]=0$.        Thus by \ref{equation splitting homological}, $H$ must lie in every splitting subalgebra $\o$ of $\p\q(n)$, so we may write $\o=\o'\rtimes\C\langle H\rangle$, where $\o'$ is splitting in $\p\s\q(n)$ by Lemma \ref{lemma semidirect product}.
	
	Let $\h$ be a Cartan subalgebra of $\p\s\q(n)$.  Then using the same arguments as in the proof of Prop.~3.13 of \cite{SSh}, we obtain that, up to conjugacy, any splitting subalgebra $\o'\sub\p\s\q(n)$ contains $\h$, and therefore is a root subalgebra of $\p\s\q(n)$.  In particular it must be stable under $H$, so that $\o'\rtimes\C\langle H\rangle$ will be a splitting subalgebra of $\p\q(n)$.  It further follows that every splitting subalgebra of $\p\q(n)$ is of this form.
	
	Therefore we obtain a bijection between the splitting subalgebras of $\p\q(n)$ and $\p\s\q(n)$ which respects inclusion and conjugation, implying the same holds for Sylow subalgebras.  We may now conclude by Prop.~3.13 of \cite{SSh}.

	\subsubsection{Defect one Kac-Moody superalgebras} The case when $\g$ is a defect one Kac-Moody superalgebra is covered by Prop.~3.12 of \cite{SSh}.
	
	\subsubsection{$\g\l(m|n)$, and $\o\s\p(m|2n)$} 
	\begin{lemma}\label{lem irred} Let $L$ be a faithful irreducible representation of
		a 0-superalgebra $\o$ such that $\o$ is a splitting subgroup in $GL(L)$. Then
		$\dim L=(1|1)$ and $\o=\s\l(1|1)$.
	\end{lemma}
	\begin{proof} By Lemma \ref {lem simplezero}, part (2), we know that
		$L$ has zero superdimension. Suppose $\dim L=(n|n)$. Let $x\in\g\l(L)$ be a generic homological element
		such that the semisimple element $h=[x,x]$ is diagonal with eigenvalues
		$a_1,\dots a_n$ linearly independent over $\mathbb Q$. By \ref{equation splitting homological}, we may assume without loss of generality that $x\in \o$. There exist a basis $\{v_1,\dots,v_n\}$ of
		$L_{\bar 0}$ and a basis $\{w_1,\dots,w_n\}$ of
		$L_{\bar 1}$ such that $hv_i=a_iv_i$, $hw_i=a_iw_i$.
		By Lemma \ref{lem simplezero}(1) we obtain that $L_{\bar 0}$ and $L_{\bar 1}$
		are isomorphic irreducible representation of $\o_{\bar 0}$ with dimension equal
		the rank of $\o_{\bar 0}$. The latter is only possible if 
		$\o_{\bar 0}\simeq\g\l(n)$ is embedded diagonally into $\g\l(L)_{\bar 0}$.
		On the other hand, $\o_{\bar 1}$ must contain self-commuting elements
		of rank $(0|1)$ and $(1|0)$. This forces $\o\simeq \s\l(n|n)$ and the latter is a zero-superalgebra only for $n=1$.
	\end{proof}    
	\begin{lemma}\label{lem classical} Let $\g=\g\l(m|n)$ or $\g=\o\s\p(m|2n)$.
		Let $d$ be the defect of $\g$ and let $\{\alpha_1,\dots,\alpha_d\}$ be a set of mutually orthogonal linearly independent isotropic roots. Let $\o$ be the Sylow subalgebra of $\g$ generated by the roots spaces
		$\g_{\pm\alpha_i}$ for $i=1,\dots,d$. Then every Sylow subalgebra of $\g$ is conjugate to $\o$.
	\end{lemma}
	\begin{proof} Let $\k\sub\g$ be a Sylow subalgebra of $\g$, and let $V$ be the standard representation of $\g$
		Let us choose a generic homological $x\in\k_{\bar 1}$ which also a generic homological element in $\g$. If $\g=\g\l(V)$, then $h$ has
		purely even or purely odd kernel and eigenvalues
		$a_1,\dots, a_n$ linearly
		independent over $\mathbb Q$. If $\g=\o\s\p(V)$, the kernel of $h$
		is purely even, purely odd, or has dimension $(1|2k)$ and eigenvalues are
		$\pm a_1,\dots,\pm a_n$ with $a_1,\dots, a_n$ linearly independent over
		$\mathbb Q$. 
		
		Let $L_1,\dots,L_k$ denote the simple non-trivial $\k$-constituents in the restriction of $V$ to $\k$. Note that each $L_i$ satisfies condition (5) of
		Lemma \ref{lem simplezero}. In particular every $L_i$ is not self dual.
		Next we notice that $0\notin\PP(L_i)$. Indeed, the zero weight space of
		$L_i$ must have superdimension zero. If $\ker h$ is purely odd or purely even,
		then clearly $0\notin\PP(L_i)$. If $\dim\ker h=(1|2k)$ and $0\in\PP(L_i)$, then using the self-duality of $V$, there exists $j\neq i$ with $L_i^*\simeq L_j$, and so $0\in\PP(L_j)$. This again is impossible by dimension constraints.
		Finally, from weight multiplicities we obtain that $\PP(L_i)\cap\PP(L_j)=\emptyset$ if $i\neq j$. Using Lemma \ref{lem simplezero} (4) 
		we obtain that $V$ is semisimple as a $\k$-module when $\ker h$ is purely even or purely odd.  If $\g=\o\s\p(V)$ with $\dim\ker h=(1|2k)$, any nontrivial extension between $\C$ and $\Pi\C$ will violate self-duality of $V$, so once again we find that $V$ is semisimple.
		
		Set
		$$\g'= \prod _{i=1}^k \g\l(L_i)\subset\g$$
		for $\g=\g\l(V)$.
		If $\g=\o\s\p(V)$ then we have $k=2l$ with $L_{i}\simeq L_{l+i}^*$, and we set
		$$\g'= \prod _{i=1}^l \g\l(L_i)\subset \g.$$
		
		Since $\k\subset\g'\subset\g$, $\k$ is splitting in $\g'$. That means
		$\k\cap\g\l(L_i)$ is splitting in $\g\l(L_i)$ by Corollary \ref{cor splitting product groups}. Now Lemma \ref{lem irred} implies $\dim L_i=(1|1)$.
		Thus, $\k$ is isomorphic the product of $SL(1|1)^d$. After conjugation, we may assume 
		$x\in \o$, and that forces $\o=\k$, as desired.
	\end{proof}
	
	\subsubsection{$\mathfrak{pe}(n)$}
	
	\begin{lemma}\label{lem periplecticirred} Let $L$ be a faithful irreducible representation of
		a 0-superalgebra $\o$ such that $L\simeq \Pi L^*$ and $\o$ is a splitting subalgebra in $\mathfrak{pe}(L)$. Then
		$\dim L=(2|2)$ and $\o\cong\mathfrak{spe}(2)$.
	\end{lemma}
	\begin{proof}By Lemma \ref {lem simplezero}, part (2), we know that
		$L$ has zero superdimension. Suppose $\dim L=(n|n)$. Let $x\in\mathfrak{pe}(L)$ be a generic homological element.
		By \ref{equation splitting homological}, we may assume without loss of generality that $x\in\o$.
		
		If $n=2k$ then the eigenvalues of the semisimple element $h=[x,x]$ are    
		$\pm a_1,\dots,\pm a_k$ such that $a_1,\dots,a_k$ are linearly independent over $\mathbb Q$.
		Let $\mathcal T$ denote the global form of $\mathbb C\langle h\rangle$ and $\mathfrak{t}=\operatorname{Lie}\mathcal T$.
		If $n=2k+1$, zero is also an eigenvalue. The dimension of every eigenspace of $h$ equals $(1|1)$
		and $h$ is a generic semisimple element in $\s\l(n)$. Since $L\cong\Pi L^*$, the center of $\o$ must act trivially, so Lemma \ref {lem simplezero}, part (1), implies
		$L_{\bar 0}$ and $\Pi L_{\bar 1}$ are isomorphic, simple $\o_{\bar 0}$-modules. Furthermore, all $\mathfrak t$-weights of $L_{\bar 0}$ have multiplicity $1$ and all non-zero weights are highest 
		weights with respect to some Borel subalgebra of $\o_{\bar 0}$. Thus, $L_{\bar 0}$ is minuscule for
		$n=2k$ and quasi-minuscule for $n=2k+1$. On the other hand, $\dim L_{\ol{0}}\leq 2\rk\o_{\ol{0}}+1$. 
		By comparing with the list of quasi-minuscule representations we see only the following possibilities
		for $\o_{\bar 0}$: $\s\l(L_{\bar 0})$, $\s\o(L_{\bar 0})$ and $\s\p(L_{\bar 0})$ (for even $n$ only).
		
		Now $\mathfrak{pe}(L)_{\bar 1}=S^2L_{\bar 0}\oplus \Lambda^2{L^*_{\bar 0}}$ decomposes into the sum of
		at most three irreducible components
		with respect to the adjoint action of $\o_{\bar 0}$. 
		
		If $\o_{\bar 1}=s\l(L_{\bar 0})$ then 
		$S^2L_{\bar 0}$ and $\Lambda^2L^*_{\bar 0}$ are irreducible and by \ref{equation splitting homological} $\o_{\bar 1}$ must contain an element from both components. Then $\mathfrak{pe}(L)_{\bar 1}=\o_{\bar 1}$
		and hence $\mathfrak{spe}(L)\subset\o$
		
		If $\o_{\bar 1}=\s\o(L_{\bar 0})$ then 
		$S^2L_{\bar 0}=R\oplus\mathbb C$ where $R$ is irreducible, and $\Lambda^2L^*_{\bar 0}$ is irreducible. By \ref{equation splitting homological} $\o_{\bar 1}$ must contain elements from both nontrivial components. 
		It is easy to see that $R$ and $\Lambda^2L^*_{\bar 0}$ generate $\mathfrak{spe}(L)$ and hence $\mathfrak{spe}(L)\subset\o$.
		
		For $\o_{\bar 1}=\s\p(L_{\bar 0})$ we can show that  $\mathfrak{spe}(L)\subset\o$ by the same argument as in the previous case. The only difference is that 
		$S^2L_{\bar 0}$ remains irreducible while $\Lambda^2L^*_{\bar 0}=R\oplus \mathbb C$.
		
		Since  $\o$ is a $0$-superalgebra and $\mathfrak{spe}(L)\subset\o$ is oddly generated, Corollary \ref{cor sub of 0 is 0} implies that $\mathfrak{spe}(L)$ is a $0$-superalgebra.   This implies that $n=1$ or $n=2$. But the case $n=1$ must be excluded as $\mathfrak{spe}(1)$ has no faithful irreducible representation.	
	\end{proof}

	\begin{lemma}\label{lem periplectic} Let $\g=\mathfrak{pe}(n)$. Then every Sylow subalgebra is conjugate to
		$\mathfrak{spe}(2)^k$ if $n=2k$ and $\mathfrak{spe}(2)^k\times \mathfrak{spe}(1)$ if $n=2k+1$.
	\end{lemma}
	\begin{proof} Let $\o$ be a Sylow subalgebra.  We choose a generic homological $x$ as in the proof of Lemma \ref{lem periplecticirred}, and assume without loss of generality that $x\in \o$. 
		
		First, we consider the case when $n$ is even. As in the proof of Lemma \ref{lem classical}
		we can conclude that the restriction of the standard $\g$-module $V$ to $\o$ is semisimple.
		Write $V$ as a direct sum of simple $\mathfrak o$-modules
		\[
		V=\bigoplus_{i=1}^m (L_i\oplus\Pi L_i^*)\oplus\bigoplus_{j=1}^l S_j,
		\]
		where $S_j\simeq\Pi S_j^*$. Now $\o$ must be splitting in
		$\prod \g\l(L_i)\times\prod \mathfrak{pe}(S_j)$. This implies $\o\cap \g\l(L_i)$ is splitting in $\g\l(L_i)$,
		hence $\dim L_i=(1|1)$, and $\o\cap\mathfrak{pe}(S_j)$ is splitting in $\mathfrak{pe}(S_j)$ hence
		$\dim S_j=(2|2)$. However, $\g\l(L_i)$ is not a splitting subgroup in $\mathfrak{pe}(L_i\oplus\Pi L_i^*)$ and therefore $m=0$. Now by Lemma \ref{lem periplectic} we obtain $\o=\prod \mathfrak{spe}(S_j)$
		which implies the statement for even $n$.
		
		The case of odd $n$ is more subtle because of the zero weight space of dimension $(1|1)$.
		In this case we have $V=V'\oplus W$ where $W$ is a $(1|1)$-dimensional $\o$-module and $V'$ is a semisimple $\o$-module. The condition
		$W\simeq \Pi W^*$ implies $\o\cap \mathfrak{pe}(W)=\mathfrak{spe}(W)$ and the proof can be finished as in the first case.
	\end{proof}
	
	This now completes the proof of Proposition \ref{prop simple conj unique}.

	\subsection{Step 2: the case when $\z(\g)_{\ol{0}}=0$} Let us now prove Theorem \ref{theorem uniqueness conjugacy} in the case when $\z(\g)_{\ol{0}}=0$.  In this case, by Thm.~6.9 of \cite{S}, we have that $\g=\c(\g)\rtimes(\r\ltimes\l)$, where $\c(\g)$ is the product of minimal ideals of $\g$, $\r$ is a reductive Lie algebra, and $\l$ is an odd abelian subalgebra with trivial action of $[\r_{\ol{0}},\r_{\ol{0}}]$.  Further, $\r\times\l$ acts faithfully by outer derivations on $\c(\g)$, i.e.~we have a natural embedding $\r\times\l\sub\operatorname{Out}(\c(\g))$.
	
	\begin{lemma}\label{lemma bijection sylows c(g) g}
		We have a bijection between Sylow subalgebras of $\c(\g)$ and Sylow subalgebras of $\g$ given by 
		\[
		\c(\g)\supseteq\o\mapsto (\o+[\l,\o_{\ol{1}}])\rtimes\l.
		\]
		This bijection preserves conjugacy classes, and so all Sylow subalgebras of $\g$ are conjugate if and only if the same is true for $\c(\g)$.
	\end{lemma}
	
	\begin{proof}
		Write $\c(\g)=\k_1\times\cdots\k_l$ as a product of minimal ideals.  By Corollary \ref{cor product}, the Sylow subalgebras of $\c(\g)$ are of the form $\o=\o_1\times\cdots\o_l$, where $\o_i\sub\k_i$ is a Sylow subalgebra.  
		
		By Lem.~5.6 of \cite{S}, $\k_i$ is either simple, odd abelian, or isomorphic to $\s\otimes\C[\xi]$ for a simple Lie algebra $\s$.  If $\k_i$ is simple, then by Proposition \ref{prop simple conj unique}, $\o_i$ is stable under $\r\ltimes\l$.  If $\k_i$ is odd abelian, then $\k_i=\o_i$ by Lem.~2.12 of \cite{SSh}, and again $\o_i$ is stable under $\r\ltimes\l$.
		
		Finally, if $\k_i\cong\s\otimes\C[\xi]$ for a simple Lie algebra $\s$, then necessarily $\o_i=\xi\s$.  However, because $\k_i$ is a minimal ideal, $\o_i+[\l,\o_i]=\k_i$.  Therefore we obtain that
		\[
		(\o+[\o,\l])\rtimes\l
		\]
		is a quasireductive subalgebra of $\g$.  We claim that it is a Sylow subalgebra.  Indeed, it is a $0$-superalgebra by direct inspection, using Theorem \ref{thm classification 0 groups}.  On the other hand it is splitting by Lemma \ref{lemma semidirect product}.  
		
		Conversely, given a Sylow subalgebra $\o'\sub\g$, it must contain $\l$ by \ref{equation splitting homological}, since $\l\sub\g_{\ol{1}}^{hom}$ and $\GG_0$ acts on $\l$ by characters.  Since $[\g_{\ol{1}},\g_{\ol{1}}]\sub\c(\g)$, we may write $\o'=\o\rtimes\l$. By Lemma \ref{lemma semidirect product}, $\o$ will be a splitting subalgebra of 
		\[
		\c(\g)\cong\k_1\times\cdots\k_l.
		\] 
		By Corollary \ref{cor product}, $\o_i:=\o\cap\k_i$ is splitting in $\k_i$, and is clearly $\l$-stable.  From this we see that $\o_i\sub\k_i$ is Sylow whenever $\k_i$ is simple or odd abelian, while $\o_i=\k_i$ when $\k_i\cong\s\otimes\C[\xi]$.  Therefore our correspondence is indeed a bijection of Sylow subalgebras.
		
		That our bijection preserves conjugacy classes is straightforward, completing the proof of the lemma.
	\end{proof}
	
	Now using step 1, it is straightforward to see that Sylow subalgebras are unique up to conjugacy in $\c(\g)$.  Therefore by Lemma \ref{lemma bijection sylows c(g) g}, Sylows are unique up to conjugacy in $\g$, completing step 2.
	
	\subsection{Final (third) step: $\g$ arbitrary} For this step we simply apply Lemma \ref{lemma cent extn even deriv} to reduce to step 2.  This completes the proof of Theorem \ref{theorem uniqueness conjugacy}.
	
	\begin{remark}
		The proof of Theorem \ref{theorem uniqueness conjugacy} in fact shows that a maximal $0$-subgroup $\KK\sub\GG$ satisfying \ref{equation splitting homological} must be Sylow.
	\end{remark}
	
	\section{Normalizers of Sylow subgroups} 
	In the final section, we give a precise description of the normalizer of a Sylow subalgebra. Afterward, we give a version of the third Sylow theorem.
	
	\subsection{Normalizer of a Sylow subalgebra}
	Let $\g$ be quasireductive with Sylow subalgebra $\g$.  We would like to describe $\n_{\g}(\o)$, the normalizer of $\o$ in $\g$.  
	
	First, it is clear that $\n_{\g}(\o)$ will contain the centre of $\g$, so we may assume that $\z(\g)_{\ol{0}}=0$.  Thus using Thm.~6.9 of \cite{S}, we may write
	\[
	\g=\c(\g)\rtimes (\l\rtimes \r),
	\]
	where $\c(\g)=\k_1\times\dots\times\k_\ell$ is the product of minimal ideals $\k_i$, $\l$ is odd abelian, and $\r=\r_{\ol{0}}$ is reductive.  As we showed in the proof of Theorem \ref{theorem uniqueness conjugacy}, we may write
	\[
	\o=\o'\rtimes \l,
	\]
	where $\o'\sub\c(\g)$ is a splitting subalgebra.  Writing $\o'_i=\k_i\cap \o$, we have 
	\[
	\o'=\o'_1\times\dots\times\o'_\ell,
	\]
	and $\o'_i=\k_i$ if $\k_i$ is either Takiff or odd abelian, while $\o'_i\sub\k_i$ is a Sylow subalgebra if $\k_i$ is simple. In the below table, we give $\n_{\g}(\o)$ where $\o\sub\g$ is a Sylow subalgebra of a simple Lie superalgebra $\g$.
	
	\renewcommand{\arraystretch}{1.5}	
	
	\begin{center}
		\begin{tabular}{|c|c|}
			\hline 
			$(\g,\o)$ & $\n_{\g}(\o)$ \\
			\hline
			$m<n$: $(\s\l(m|n),\s\l(1|1)^m)$ & $\s(\g\l(1|1)^m\times\g\l(n-m))$\\
			\hline         $(\p\s\l(n|n),\p(\s\l(1|1)^n))$ & $\p\s(\g\l(1|1)^n)$ \\
			\hline 
			$m> 2n$: $(\o\s\p(m|2n),\s\l(1|1)^{n})$ & $\g\l(1|1)^{n}\times \s\o(m-2n)$ \\
			\hline 
			$2m\leq 2n$: $(\o\s\p(2m|2n),\s\l(1|1)^{m})$ & $\g\l(1|1)^{m}\times \s\p(2(n-m))$ \\
			\hline 
			$2m+1\leq 2n$: $(\o\s\p(2m+1|2n),\s\l(1|1)^{m})$ & $\g\l(1|1)^{m}\times \o\s\p(1|2(n-m))$ \\
			\hline 
			$(\mathfrak{d}(2,1;\alpha),\s\l(1|1))$ & $\g\l(1|1)\times\C$ \\ 
			\hline 
			$(\mathfrak{ag}(1|2),\s\l(1|1))$ & $\g\l(1|1)\times \s\l(2)$ \\
			\hline 
			$(\mathfrak{ab}(1|3),\s\l(1|1))$ & $\g\l(1|1)\times\s\l(3)$ \\
			\hline 
			$(\p\s\q(2n),\p\s(\q(2)^n))$ & $\p\s(\q(2)^n)$ \\
			\hline
			$(\p\s\q(2n+1),\p\s(\q(2)^n\times\q(1)))$ & $\p\s(\q(2)^n\times\q(1))$ \\
			\hline
			$(\s\mathfrak{pe}(2n),\s\mathfrak{pe}(2)^n)$ & $\s(\mathfrak{pe}(2)^n)$  \\
			\hline
			$(\s\mathfrak{pe}(2n+1),\s\mathfrak{pe}(2)^n\times\s\mathfrak{pe}(1))$ & $\s(\mathfrak{pe}(2)^n\times\mathfrak{pe}(1))$ \\
			\hline         
		\end{tabular}
	\end{center}
	
	\begin{lemma}
		$\n_{\c(\g)}(\o')$ is quasireductive and stable under $\l\rtimes\r$.
	\end{lemma}
	\begin{proof}
		The fact that $\n_{\c(\g)}(\o')$ is quasireductive is immediate from the above table.  By Lemma \ref{lemma derivations simples}, $\o'$ is stable under $\l\times \r$, and thus the same must be true of $\n_{\c(\g)}(\o')$.  
	\end{proof}

	\begin{prop}\label{prop normalizer description}
		Let $\g$ be quasireductive with $\o\sub\g$ a Sylow subalgebra.  Then $\z(\g)_{\ol{0}}\sub\n_{\g}(\o)$, and we have
		\[
		\n_{\g}(\o)/\z(\g)_{\ol{0}}=\n_{\c(\g)}(\o')\rtimes(\l\rtimes\r).
		\]
		where $\o'=\left(\o/(\z(\g)_{\ol{0}}\cap\o)\right)\cap\c(\g)$.
		
	\end{prop}
	\begin{proof}
		This is clear from the previous lemma.
	\end{proof}
	
	\begin{cor}\label{cor normalizer qred hom elts}
		Let $\o\sub\g$ be a Sylow subalgebra, and write $\n_{\g}(\o)$ for its normalizer in $\g$. Then we have:
		\begin{enumerate}
			\item $\n_{\g}(\o)$ is quasireductive in $\g$;
			\item if $\k=\n_{\g}(\o)/\o$, then $\Rep_{\k_{\ol{0}}}\k$ is semisimple;
			\item $\n_{\g}(\o)^{hom}_{\ol{1}}=\o_{\ol{1}}^{hom}$.    
		\end{enumerate} 
	\end{cor}
	
	\begin{proof}
		Part (1) follows from Proposition \ref{prop normalizer description}. Part (2) is because the quotient will have the trivial superalgebra as a splitting subalgebra.  Part (3) then follows from Proposition \ref{prop semisimple criterion}.
	\end{proof}
	
	For completeness and future use, we also record here the table which describes, for $\g$ simple, the algebraic group of inner automorphisms of $\g$ which stabilize a Sylow subalgebra $\o\sub\g$.  We write this algebraic group as $\NN_0$.

	\renewcommand{\arraystretch}{1.5}	
	
	\begin{center}
		\begin{tabular}{|c|c|}
			\hline 
			$(\g,\o)$ & $\NN_{0}$\\
			\hline 
			$m<n$: & $S(GL(1|1)^m_0\times GL(n-m))\rtimes S_m$\\
			$(\s\l(m|n),\s\l(1|1)^m)$ & \\
			\hline 
			$(\p\s\l(n|n),\p(\s\l(1|1)^n))$  & $PS(GL(1|1)^n_0)\rtimes S_n$\\
			\hline  
			$m> 2n$: & $\left[GL(1|1)^{n}_0\times SO(m-2n)\right]\rtimes((\Z/2\Z)^n\rtimes S_n)$\\
			$(\o\s\p(m|2n),\s\l(1|1)^n)$ & \\
			\hline 
			$2m\leq 2n$: & $\left[GL(1|1)^m_0\times Sp(2(n-m)\right]\rtimes ((\Z/2\Z)^{m-1}\rtimes S_m)$\\
			$(\o\s\p(2m|2n),\s\l(1|1)^{m})$ & \\
			\hline 
			$2m+1\leq 2n$:  & $\left[GL(1|1)^{m}_0\times Sp(2(n-m)\right]\rtimes((\Z/2\Z)^m\rtimes S_m)$\\
			$(\o\s\p(2m+1|2n),\s\l(1|1)^{m})$  & \\
			\hline 
			$(\mathfrak{d}(2,1;\alpha),\s\l(1|1))$ & $\G_m^3\rtimes \Z/2\Z$\\ 
			\hline 
			$(\mathfrak{ag}(1|2),\s\l(1|1))$  & $\left[\G_m^2\times SL(2)\right]\rtimes \Z/2\Z$\\
			\hline 
			$(\mathfrak{ab}(1|3),\s\l(1|1))$ & $\left[\G_m^2\times SL(3)\right]\rtimes \Z/2\Z$\\
			\hline 
			$(\p\s\q(2n),\p\s(\q(2)^n))$ & $PS(GL(2)^n)\rtimes S_n$\\
			\hline
			$(\p\s\q(2n+1),\p\s(\q(2)^n\times\q(1)))$ & $PS(GL(2)^n\times GL(1))\rtimes S_n$\\
			\hline
			$(\s\mathfrak{pe}(2n),\s\mathfrak{pe}(2)^n)$ & $S(GL(2)^n)\rtimes S_n$ \\
			\hline
			$(\s\mathfrak{pe}(2n+1),\s\mathfrak{pe}(2)^n\times\s\mathfrak{pe}(1))$ & $S(GL(2)^n\times GL(1))\rtimes S_n$\\
			\hline         
		\end{tabular}
	\end{center}
	
	\

	\begin{cor}
		If $\GG$ is quasireductive with Sylow subgroup $\OO\sub\GG$, then $\NN_{\GG}(\OO)$ is once again quasireductive, and $\Rep\NN_{\GG}(\OO)/\OO$ is semisimple.
	\end{cor}
	
	The following result is extremely useful.  Ideally, one could find a simple proof which doesn't rely on the explicit description of Sylow subalgebras.
	\begin{cor}\label{cor cent in norm}
		Let $\g$ be quasireductive with Sylow subalgebra $\o\sub\g$. Let $\mathfrak{t}\sub\o_{\ol{0}}$ be a maximal torus.  Then $\c_{\g}(\mathfrak{t})\sub\n_{\g}(\o)$.  
	\end{cor}
	\begin{proof}
		This follows from our explicit description of $\n_{\g}(\o)$ given in Proposition \ref{prop normalizer description}.
	\end{proof}
	
	\begin{cor}\label{cor gen elt}
		Let $\GG$ be quasireductive with Sylow subgroup $\OO\sub\GG$.  Write $\TT\sub\OO_0$ for a maximal torus of $\OO_0$.  Then there exists $x\in\o_{\ol{1}}^{hom}$ such that $\CC_{\GG}(x^2)=\CC_{\GG}(\TT).$  In particular $\CC_{\GG}(x^2)\sub\NN_{\GG}(\OO)$.
	\end{cor}
	
	\begin{definition}
		We call an element $x\in\g_{\ol{1}}^{hom}$ generic if for some Sylow subgroup $\OO\sub\GG$, $x\in\operatorname{Lie}\OO$ and satisfies the conditions of Corollary \ref{cor gen elt}.
	\end{definition}
	
	\begin{proof}[Proof of Corollary \ref{cor gen elt}]
		The second statement follows from the connectedness of centralizers of tori and Corollary \ref{cor cent in norm}. 
		
		For the first statement, write $\g=\operatorname{Lie}\GG$.  Clearly we may assume that $\mathfrak{z}(\g)_{\ol{0}}=0$, and so we do assume this. As usual, write $\g=\c(\g)\rtimes(\l\rtimes\r)$, so that $\o=\o'\rtimes\l$, where $\o=\operatorname{Lie}\OO$.     Writing $\mathfrak{t}=\operatorname{Lie}\TT$, it is clear that $\mathfrak{t}\sub\o'$, so we only need to find such an element $x$ when $\g$ is either simple or a Takiff 0-superalgebra. 
		These cases are easily checked, case by case.
	\end{proof}
	
	\begin{lemma}
		The set of generic elements of $\g_{\ol{1}}$ form a $\Aut(\g)$-stable subset of $\g_{\ol{1}}^{hom}$.
	\end{lemma}
	
	\begin{proof}
		The property of being generic is clearly preserved under any automorphism, so the result is clear.
	\end{proof}
	
	\begin{lemma}
		Let $\g$ be quasireductive with Sylow subalgebra $\o$.  Then the set of generic elements in $\o_{\ol{1}}$ is open and dense in $\o_{\ol{1}}$.
	\end{lemma}
	\begin{proof}
		Let $V\sub\o_{\ol{0}}$ be the open subset of regular semisimple elements $t$ for which \linebreak $\c_{\g}(t)=\c_{\g}(\c_{\o}(t))$. Then the set of generic elements is exactly the preimage of $V$ under the squaring map $\o_{\ol{1}}\to\o_{\ol{0}}$.  By Corollary \ref{cor gen elt} this preimage is nonempty, and so we are done.
	\end{proof}
	
	\begin{cor}
		Let $\g$ be quasireductive with Sylow subalgebras $\o,\o'$ such that $\o_{\ol{0}}=\o_{\ol{0}}'$.  Then $\o=\o'$.
	\end{cor}
	
	\begin{proof}
		Let $\mathfrak{t}\sub\o_{\ol{0}}$ be a maximal torus, and let $\h$ denote the centralizer of $\mathfrak{t}$ in $\o$.  Then by Corollary \ref{cor cent in norm}, 
		\[
		\h\sub\c_{\g}(\mathfrak{t})\sub \n_{\g}(\o').
		\]
		Since $\h_{\ol{1}}=\h_{\ol{1}}^{hom}$, by Corollary \ref{cor normalizer qred hom elts} we must have $\h\sub\o'$.  Now we apply Lemma \ref{lem rootzero} to obtain: 
		\[
		\o=\h+[\h,\o_{\ol{0}}]=\o'.
		\]
	\end{proof}
	
	\begin{cor}
		Let $\GG$ be quasireductive with Sylow subgroups $\OO,\OO'\sub\GG$.  If $\OO_0=\OO_0'$, then $\OO=\OO'$.
	\end{cor}
	
	\subsection{The third Sylow Theorem}\label{section third sylow thm}
	
	Let $G$ be a finite group, and write $\text{Syl}_p(G)$ for the set of Sylow $p$-subgroups of $G$.  The third Sylow theorem states that $|\text{Syl}_p(G)|$ is congruent to 1 
	modulo $p$.  Since it is clear that $|\text{Syl}_p(G)|=|G/N_G(P)|$ for any Sylow subgroup $P$, we may view this as saying that $\vol(G/N_G(P))=1$ in the base field.
	
	We would like to give an analogous result in the super setting.  One statement is that $\vol(\GG/\NN_\GG(\OO))\neq0$ for any Sylow subgroup $\OO\sub\GG$; however this is obvious by \ref{lemma nonzero} because $\NN_\GG(\OO)$ contains $\OO$, and thus is splitting.
	
	We formulate a slightly different claim which is quite useful.  For this, recall that if $x\in\g_{\ol{1}}^{hom}$, then we have the Duflo-Serganova functor $DS_x:\Rep_{\g_{\ol{0}}}\g\to s\operatorname{Vec}$ given by
	\[
	DS_xM=\frac{\operatorname{Ker}(x:M\to M)}{\operatorname{Im}(x:M\to M)\cap\operatorname{Ker}(x:M\to M)}.
	\]
	We refer to \cite{GHSS} for more details about this functor.  
	
	If $A$ is a supercommutative algebra in $\Rep_{\g_{\ol{0}}}\g$, then a straightforward check shows that $DS_xA$ will once again be a supercommutative algebra.

	\begin{definition}
		Let $\GG$ be quasireductive with Sylow subgroup $\OO\sub\GG$, and let $\TT\sub\OO_0$ be a maximal torus in $\OO_0$.  Define $W_{\GG}:=\NN_{\GG}(\TT)/\CC_{\GG}(\TT)$, and $W_\OO:=\NN_{\NN_{\GG}(\OO)}(\TT)/\CC_{\NN_{\GG}(\OO)}(\TT)$.
	\end{definition}
	
	\begin{lemma}
		Both $W_{\GG}$ and $W_{\NN_{\GG}(\OO)}$ are finite groups, and $W_{\NN_{\GG}(\OO)}$ is naturally a subgroup of $W_{\GG}$.
	\end{lemma}
	\begin{proof}
		That these groups are finite is a consequence of the following statement from the theory of reductive groups: the normalizer modulo the centralizer of any torus is always a finite group.  It is clear that $W_{\NN_{\GG}(\OO)}$ is a subgroup of $W_{\GG}$.
	\end{proof}
	
	\begin{thm}
		Let $\GG$ be quasireductive with Sylow subgroup $\OO$, and let $x\in\o_{\ol{1}}^{hom}$ be generic.  Then we have an isomorphism of algebras
		\[
		DS_x\C[\GG/\NN_{\GG}(\OO)]=\C[W_{\GG}/W_{\NN_{\GG}(\OO)}],
		\]
		induced by the inclusion of spaces $W_{\GG}/W_{\NN_{\GG}(\OO)}\sub \GG/\NN_{\GG}(\OO)$.
	\end{thm}
	
	\begin{proof}
		Write $\TT=\CC_{\GG}(x^2)\sub\OO_0$, which is a maximal torus of $\OO_0$ by definition of generic element.
		
		Suppose that $x$ vanishes at some closed point $g\NN_{\GG}(\OO)$ for $g\in\GG(\C)$.  Then $\Ad(g)(x)\in\n_{\g}(\o)^{hom}_{\ol{1}}=\o_{\ol{1}}^{hom}$, where the last equality is by Corollary \ref{cor normalizer qred hom elts}.  If we write $\TT'=\CC_{\GG}(\Ad(g)(x^2))$, then $\TT'$ will also be a maximal torus of $\OO_0$, so after multiplying by some element of $\OO(\C)$ we may assume that $\TT'=\TT$, and thus $g\in\NN_{\GG}(\TT)(\C)$.  If $g\in\CC_{\GG}(\TT)$, then by Corollary \ref{cor cent in norm} we have $g\in\NN_{\GG}(\OO)(\C)$.  
		
		Thus we have shown that the zeroes of $x$ lie in the image of the natural map $W_{\GG}\to \GG/\NN_{\GG}(\OO)$, and from this it is clear that the zeroes are exactly given by $W_{\GG}/W_{\NN_{\GG}(\OO)}$.  
		
		Finally, we observe that for every zero $z$ of $x$, the endomorphism $[x,-]$ of $T_{z}^*\GG/\NN_{\GG}(\OO)$ is an isomorphism, once again by Corollary \ref{cor cent in norm} and the definition of generic elements.  Thus we may conclude by Thm.~4.1 of \cite{SSh2}.
	\end{proof}
	
	\subsection{Table of inclusions $W_{\NN_{\GG}(\OO)}\sub W_{\GG}$}\label{section table} We finish by giving a table of the groups $W_{\NN_{\GG}(\OO)}$ and $W_{\GG}$ for each simple superalgebra.  Observe that $W_{\NN_{\GG}(\OO)}=W_{\GG}$ unless $\g$ is of type $\q$.

	\renewcommand{\arraystretch}{1.5}	
	
	\begin{center}
		\begin{tabular}{|c|c|}
			\hline 
			$(\g,\o)$ & $W_{\NN_{\GG}(\OO)}\sub W_{\GG}$ \\
			\hline
			$m<n$: $(\s\l(m|n),\s\l(1|1)^m)$ & $W_{\NN_{\GG}(\OO)}=W_{\GG}=S_m$\\
			\hline         $(\p\s\l(n|n),\p(\s\l(1|1)^n))$ & $W_{\NN_{\GG}(\OO)}=W_{\GG}=S_n$ \\
			\hline 
			$m> 2n$: $(\o\s\p(m|2n),\s\l(1|1)^{n})$ & $W_{\NN_{\GG}(\OO)}=W_{\GG}=(\Z/2\Z)^n\rtimes S_n$ \\
			\hline 
			$2m\leq 2n$: $(\o\s\p(2m|2n),\s\l(1|1)^{m})$ & $W_{\NN_{\GG}(\OO)}=W_{\GG}=(\Z/2\Z)^{m-1}\rtimes S_m$ \\
			\hline 
			$2m+1\leq 2n$: $(\o\s\p(2m+1|2n),\s\l(1|1)^{m})$ & $W_{\NN_{\GG}(\OO)}=W_{\GG}=(\Z/2\Z)^m\rtimes S_m$ \\
			\hline 
			$(\mathfrak{d}(2,1;\alpha),\s\l(1|1))$ & $W_{\NN_{\GG}(\OO)}=W_{\GG}=\Z/2\Z$ \\ 
			\hline 
			$(\mathfrak{ag}(1|2),\s\l(1|1))$ & $W_{\NN_{\GG}(\OO)}=W_{\GG}=\Z/2\Z$ \\
			\hline 
			$(\mathfrak{ab}(1|3),\s\l(1|1))$ & $W_{\NN_{\GG}(\OO)}=W_{\GG}=\Z/2\Z$ \\
			\hline 
			$(\p\s\q(2n),\p\s(\q(2)^n))$ & $W_{\NN_{\GG}(\OO)}=(S_2)^n\rtimes S_n\sub S_{2n}=W_{\GG}$ \\
			\hline
			$(\p\s\q(2n+1),\p\s(\q(2)^n\times\q(1)))$ & $W_{\NN_{\GG}(\OO)}=(S_2)^n\rtimes S_n\sub S_{2n+1}=W_{\GG}$ \\
			\hline
			$(\s\mathfrak{pe}(2n),\s\mathfrak{pe}(2)^n)$ &  $W_{\NN_{\GG}(\OO)}=W_{\GG}=(S_2)^n\rtimes S_n$ \\
			\hline
			$(\s\mathfrak{pe}(2n+1),\s\mathfrak{pe}(2)^n\times\s\mathfrak{pe}(1))$ & $W_{\NN_{\GG}(\OO)}=W_{\GG}=(S_2)^n\rtimes S_n$ \\
			\hline         
		\end{tabular}
	\end{center}

	\bibliographystyle{amsalpha}

\begin{thebibliography}{99999999}
		
		\bibitem{AH} A.~Alldridge and J.~Hilgert, \emph{Invariant Berezin integration on homogeneous supermanifolds}. Journal of Lie Theory, Vol.~20 (2010): 65–91.
		
		
		\bibitem{Borelcompact} A.~Borel, \emph{Semisimple groups and Riemannian symmetric spaces}. New Delhi: Hindustan Book Agency, Vol.~16 (1998).
		
		\bibitem{BKN} B.~Boe, J.~Kujawa, and D.~Nakano, \emph{Cohomology and support varieties for Lie superalgebras}. 
		Transactions of the American Mathematical Society, 362.12 (2010): 6551-6590.
		
		\bibitem{BKN2} B.~Boe, J.~Kujawa, and D.~Nakano, \emph{Tensor triangular geometry for classical Lie superalgebras}. 
		Advances in Mathematics, Vol.~314 (2017): 228-277.
		
		\bibitem{CF} C.~Carmeli and R.~Fioresi, \emph{Superdistributions, analytic and algebraic super Harish-Chandra pairs}. Pacific Journal of Mathematics, 263.1 (2013): 29-51.
		
		\bibitem{D} P.~Deligne, \emph{Catégories tensorielles}. Moscow Mathematical Journal, 2.2 (2002): 227-248.
		
		\bibitem{DMor} P.~Deligne and J.~W.~Morgan, \emph{Notes On Supersymmetry (following Joseph Bernstein)}. Quantum Fields And Strings: A Course For Mathematicians, American Mathematical Society, Vol.~1 (1999).
		
		\bibitem{ES} I.~Entova-Aizenbud and V.~Serganova, \emph{Jacobson-Morozov Lemma for algebraic supergroups}. Advances in Mathematics, Vol.~398 (2022): 108240.
		
		\bibitem{G} M.~Gorelik, \emph{Depths and cores in the light of DS-functors}. Preprint arXiv:2010.05721 (2020).
		
		\bibitem{GHSS} M.~Gorelik, C.~Hoyt, V.~Serganova and A.~Sherman, \emph{The Duflo–Serganova functor, vingt ans aprés}. Journal of the Indian Institute of Science, 102.3 (2022): 961-1000.
		
		\bibitem{K} V.~Kac, \emph{Lie superalgebras}. Advances in Mathematics, 26.1 (1977): 8-96.
		
		\bibitem{Manin} Y.I.~Manin, \emph{Gauge field theory and complex geometry}. Springer Science \& Business Media, Vol.~289 (2013).
		
		\bibitem{MT} A.~Masuoka, Y.~Takahashi, \emph{Geometric construction of quotients $G/H$ in supersymmetry}.	 Transformation Groups, 26.1 (2021): 347–375.
		
		\bibitem{N} G.~Navarro, \emph{Character theory and the McKay conjecture}. Cambridge University Press, Vol.~175 (2018).
		
		\bibitem{R} A.~Rogers, \emph{Supermanifolds: theory and applications}. World Scientific (2007).
		
		\bibitem{S} V.~Serganova, \emph{Quasireductive supergroups}. New developments in Lie theory and its applications, Vol.~544 (2011): 141-159.
		
		\bibitem{SSh} V.~Serganova and A.~Sherman \emph{Splitting quasireductive supergroups and volumes of supergrassmannians}. To appear in Algebraic Geometry and Physics.
		
		\bibitem{SSh2} V.~Serganova and A.~Sherman, \emph{Localization theorem for homological vector fields}. Communications in Mathematical Physics, Vol.~405, no.~39 (2024).
		
		\bibitem{SV} V.~Serganova and D.~Vaintrob, \emph{Localization for CS manifolds and volume of homogeneous superspaces}. Preprint arXiv:2212.07503 (2022).
		
		\bibitem{ShSi} A.~Sherman and L.~Silberberg, \emph{A queer Kac-Moody construction}.  Preprint arXiv:2309.09559 (2023).
		
		\bibitem{V} D. Vaintrob, \emph{Analysis on real and complex supermanifolds}. In preparation.
		
		\bibitem{W} E.~Witten, \emph{Notes on supermanifolds and integration}. Pure
		and Applied Mathematics Quarterly, 15.1 (2019): 3-56
		
	\end{thebibliography}
	
\end{document}